\documentclass[12pt]{amsart}      
\usepackage{amssymb}
\usepackage{eucal}
\usepackage{amsmath}
\usepackage{amscd}
\usepackage[all]{xy}           

\usepackage{amsfonts,latexsym}
\usepackage{xspace}
\usepackage{hyperref}
\usepackage{float}
\usepackage{color}
\usepackage{colordvi}
\usepackage{multicol}
\usepackage{multirow}

\newcommand{\nc}{\newcommand}
\newcommand{\delete}[1]{}

\nc{\dfootnote}[1]{{}}          
\nc{\ffootnote}[1]{\dfootnote{#1}}

\delete{
\nc{\mfootnote}[1]{{}}        
\nc{\ofootnote}[1]{{}}        
}

\nc{\mfootnote}[1]{\footnote{#1}} 
\nc{\ofootnote}[1]{\footnote{\tiny Older version: #1}} 

\nc{\mlabel}[1]{\label{#1}}  
\nc{\mcite}[1]{\cite{#1}}  
\nc{\mref}[1]{\ref{#1}}  
\nc{\mkeep}[1]{{}}      
\nc{\mbibitem}[1]{\bibitem{#1}} 

\delete{
\nc{\mcite}[1]{\cite{#1}{{\bf{{\ }(#1)}}}}  
\nc{\mlabel}[1]{\label{#1}  
{\hfill \hspace{1cm}{\bf{{\ }\hfill(#1)}}}}
\nc{\mref}[1]{\ref{#1}{{\bf{{\ }(#1)}}}}  
\nc{\mbibitem}[1]{\bibitem[\bf #1]{#1}} 
\nc{\mkeep}[1]{\marginpar{{\bf #1}}} 
}

\newtheorem{theorem}{Theorem}[section]
\newtheorem{prop}[theorem]{Proposition}
\newtheorem{defn}[theorem]{Definition}
\newtheorem{lemma}[theorem]{Lemma}
\newtheorem{coro}[theorem]{Corollary}
\newtheorem{prop-def}{Proposition-Definition}[section]

\newtheorem{remark}[theorem]{Remark}

\nc{\bond}{\vdash}

\nc{\comp}[1]{\langle #1\rangle}
\nc{\spr}{\cdot}
\nc{\disp}[1]{\displaystyle{#1}}
\nc{\bin}[2]{ (_{\stackrel{\scs{#1}}{\scs{#2}}})}  
\nc{\binc}[2]{ \left (\!\! \begin{array}{c} \scs{#1}\\
    \scs{#2} \end{array}\!\! \right )}  
\nc{\bbinc}[2]{ \left (\!\! \begin{array}{c} {#1}\\
    {#2} \end{array}\!\! \right )}  
\nc{\bincc}[2]{  \left ( {\scs{#1} \atop
    \vspace{-.5cm}\scs{#2}} \right )}  
\nc{\sarray}[2]{\begin{array}{c}#1 \vspace{.1cm}\\ \hline
    \vspace{-.35cm} \\ #2 \end{array}}
\nc{\bs}{\bar{S}}
\nc{\dcup}{\stackrel{\bullet}{\cup}}
\nc{\dbigcup}{\stackrel{\bullet}{\bigcup}}

\nc{\la}{\longrightarrow}
\nc{\fe}{\'{e}}
\nc{\rar}{\rightarrow}
\nc{\dar}{\downarrow}
\nc{\dap}[1]{\downarrow \rlap{$\scriptstyle{#1}$}}
\nc{\uap}[1]{\uparrow \rlap{$\scriptstyle{#1}$}}
\nc{\defeq}{\stackrel{\rm def}{=}}
\nc{\dis}[1]{\displaystyle{#1}}

\nc{\hcm}{\ \hat{,}\ }
\nc{\hcirc}{\hat{\circ}}
\nc{\hts}{\hat{\shpr}}
\nc{\lts}{\stackrel{\leftarrow}{\shpr}}
\nc{\rts}{\stackrel{\rightarrow}{\shpr}}
\nc{\lleft}{[}
\nc{\lright}{]}
\nc{\uni}[1]{\tilde{#1}}
\nc{\free}[1]{\bar{#1}}
\nc{\den}[1]{\check{#1}}
\nc{\lrpa}{\wr}
\nc{\curlyl}{\left \{ \begin{array}{c} {} \\ {} \end{array}
    \right . \!\!\!\!\!\!\!}
\nc{\curlyr}{ \!\!\!\!\!\!\!
    \left . \begin{array}{c} {} \\ {} \end{array}
    \right \} }
\nc{\longmid}{\left | \begin{array}{c} {} \\ {} \end{array}
    \right . \!\!\!\!\!\!\!}
\nc{\ot}{\otimes}
\nc{\bigot}{\bigotimes}
\nc{\mdiv}{\mrm{div}}
\nc{\sg}{S}
\nc{\ig}{I}
\nc{\pg}{P}
\nc{\jg}{J}
\nc{\eg}{E}
\nc{\fg}{F}
\nc{\cg}{C}
\nc{\mg}{M}
\nc{\abg}{C} \nc{\bas}{B}
\nc{\Lyn}{\mrm{Lyn}}
\nc{\lyn}{\Lyn}
\nc{\sG}{{\cals}}
\nc{\iG}{{\cali}}
\nc{\jG}{{\calj}}
\nc{\eG}{{\cale}}
\nc{\pG}{{\calp}}
\nc{\fG}{{\calf}}
\nc{\cG}{{\calc}}
\nc{\mG}{{\calm}}
\nc{\ora}[1]{\stackrel{#1}{\rar}}
\nc{\ola}[1]{\stackrel{#1}{\la}}
\nc{\pex}[1]{\{#1\}}
\nc{\scs}[1]{\scriptstyle{#1}}
\nc{\mrm}[1]{{\rm #1}}
\nc{\sym}[1]{{\widehat{#1}}}
\nc{\margin}[1]{\marginpar{\rm #1}}   
\DeclareMathOperator*{\dirlim}{\displaystyle{\lim_{\longrightarrow}}}
\nc{\mvp}{\vspace{0.5cm}}
\nc{\svp}{\vspace{2cm}}
\nc{\vp}{\vspace{8cm}}
\nc{\proofbegin}{\noindent{\bf Proof: }}
\nc{\proofend}{\qed}
\font\cyr=wncyr10
\nc{\sha}{{\mbox{\cyr X}}}  
\newfont{\scyr}{wncyr10 scaled 550}
\nc{\ssha}{\mbox{\bf \scyr X}}
\newfont{\bcyr}{wncyr10 scaled 1000}
\nc{\ncsha}{{\mbox{\cyr X}^{\mathrm NC}}}
\nc{\ncshao}{{\mbox{\cyr X}^{\mathrm NC,\,0}}}
\nc{\shpr}{\diamond}    
\DeclareMathOperator*{\bigshpr}{\Large{\diamond}}
\DeclareMathOperator*{\bigshprl}{\Large{\Diamond}_\lambda}
\DeclareMathOperator*{\bigsot}{\ot}
\nc{\shprl}{{{\shpr}_\lambda}}
\nc{\shpro}{\diamond^0}    
\nc{\shpru}{\check{\diamond}}
\nc{\catpr}{\diamond_l}
\nc{\rcatpr}{\diamond_r}
\nc{\lapr}{\diamond_a}
\nc{\lepr}{\diamond_e}
\nc{\tcon}{^{\ot}}
\nc{\lengord}{_{\mrm{leng}}}
\nc{\lexord}{_{\mrm{lex}}}
\nc{\com}{^{\mrm{c}}}
\nc{\vep}{\varepsilon}
\nc{\labs}{\mid\!}
\nc{\lambdaw}{}
\nc{\rabs}{\!\mid}
\nc{\hsha}{\widehat{\sha}}
\nc{\lsha}{\stackrel{\leftarrow}{\sha}}
\nc{\rsha}{\stackrel{\rightarrow}{\sha}}
\nc{\lc}{[}
\nc{\rc}{]}
\nc{\rbset}{R}
\nc{\rbnum}{r}
\nc{\rbfun}{\mathbf{R}}
\nc{\pset}{P}
\nc{\pnum}{p}
\nc{\pfun}{\mathbf{P}}
\nc{\spset}{SP}
\nc{\spnum}{sp}
\nc{\spgen}{\mathbf{SP}}
\nc{\srbi}[1]{\{#1\}}

\nc{\ann}{\mrm{ann}}
\nc{\Aut}{\mrm{Aut}}
\nc{\can}{\mrm{can}}
\nc{\colim}{\mrm{colim}}
\nc{\Cont}{\mrm{Cont}}
\nc{\rchar}{\mrm{char}}
\nc{\cok}{\mrm{coker}}
\nc{\dtf}{{R-{\rm tf}}}
\nc{\dtor}{{R-{\rm tor}}}

\nc{\Div}{{\mrm Div}}
\nc{\End}{\mrm{End}}
\nc{\Ext}{\mrm{Ext}}
\nc{\Fil}{\mrm{Fil}}
\nc{\Frob}{\mrm{Frob}}
\nc{\Gal}{\mrm{Gal}}
\nc{\GL}{\mrm{GL}}
\nc{\lord}{\mrm{L-order}\xspace}
\nc{\rme}{\mrm{E}}
\nc{\rml}{\mrm{L}}
\nc{\rmt}{\mrm{T}}
\nc{\Sym}{\mrm{Sym}}
\nc{\Hom}{\mrm{Hom}}
\nc{\hsr}{\mrm{H}}
\nc{\hpol}{\mrm{HP}}
\nc{\id}{\mrm{id}}
\nc{\im}{\mrm{im}}
\nc{\incl}{\mrm{incl}}
\nc{\length}{\mrm{length}}
\nc{\leng}{\mrm{\ell}}
\nc{\LR}{\mrm{LR}}
\nc{\mchar}{\rm char}
\nc{\NC}{\mrm{NC}}
\nc{\mpart}{\mrm{part}}
\nc{\os}{\mrm{OS}}
\nc{\qs}{\mrm{QS}}
\nc{\ql}{{\QQ_\ell}}
\nc{\qp}{{\QQ_p}}
\nc{\rank}{\mrm{rank}}
\nc{\rcot}{\mrm{cot}}
\nc{\rdef}{\mrm{def}}
\nc{\rdiv}{{\rm div}}
\nc{\rtf}{{\rm tf}}
\nc{\rtor}{{\rm tor}}
\nc{\res}{\mrm{res}}
\nc{\sh}{\mrm{MS}}
\nc{\Spec}{\mrm{Spec}}
\nc{\tor}{\mrm{tor}}
\nc{\Tr}{\mrm{Tr}}
\nc{\tr}{\mrm{tr}}
\nc{\TL}{\mrm{TL}}
\nc{\TEL}{\mathrm{TEL}}
\nc{\ETL}{\mathrm{ETL}}
\nc{\EL}{\mathrm{EL}}
\nc{\RETL}{\mathrm{RETL}}
\nc{\EETL}{\widetilde{\TL}}
\nc{\TLs}{\sym{\mrm{TL}}}
\nc{\TELs}{\sym{\mathrm{TEL}}}
\nc{\ETLs}{\sym{\mathrm{ETL}}}
\nc{\ELs}{\sym{\mathrm{EL}}}
\nc{\RETLs}{\sym{\mathrm{RETL}}}
\nc{\EETLs}{\sym{\widetilde{\TL}}}
\nc{\word}{\rm word\xspace}
\nc{\words}{\rm words\xspace}
\nc{\varab}{\phi_{\alpha,\beta}}

\nc{\ab}{\mathbf{Ab}}
\nc{\Alg}{\mathbf{Alg}}
\nc{\Algo}{\mathbf{Alg}^0}
\nc{\Bax}{\mathbf{Bax}}
\nc{\Baxo}{\mathbf{Bax}^0}
\nc{\RBo}{\mathbf{RB}^0}
\nc{\BRB}{\mathbf{RB}}
\nc{\Dend}{\mathbf{DD}}
\nc{\bfk}{{\bf k}}
\nc{\bfone}{{\bf 1}}
\nc{\base}[1]{{a_{#1}}}
\nc{\detail}{\marginpar{\bf More detail}
    \noindent{\bf Need more detail!}
    \svp}
\nc{\Diff}{\mathbf{Diff}}
\nc{\gap}{\marginpar{\bf Incomplete}\noindent{\bf Incomplete!!}
    \svp}
\nc{\FMod}{\mathbf{FMod}}
\nc{\RB}{\mathbf{RB}}
\nc{\Int}{\mathbf{Int}}
\nc{\Mon}{\mathbf{Mon}}
\nc{\remarks}{\noindent{\bf Remarks: }}
\nc{\Rep}{\mathbf{Rep}}
\nc{\Rings}{\mathbf{Rings}}
\nc{\Sets}{\mathbf{Sets}}
\nc{\DT}{\mathbf{DT}}

\nc{\BA}{{\Bbb A}}
\nc{\CC}{{\Bbb C}}
\nc{\DD}{{\Bbb D}}
\nc{\EE}{{\Bbb E}}
\nc{\FF}{{\Bbb F}}
\nc{\GG}{{\Bbb G}}
\nc{\HH}{{\Bbb H}}
\nc{\LL}{{\Bbb L}}
\nc{\NN}{{\Bbb N}}
\nc{\QQ}{{\Bbb Q}}
\nc{\RR}{{\Bbb R}}
\nc{\TT}{{\Bbb T}}
\nc{\VV}{{\Bbb V}}
\nc{\ZZ}{{\Bbb Z}}

\nc{\cala}{{\mathcal A}}
\nc{\calc}{{\mathcal C}}
\nc{\cald}{{\mathcal D}}
\nc{\cale}{{\mathcal E}}
\nc{\calf}{{\mathcal F}}
\nc{\calg}{{\mathcal G}}
\nc{\calh}{{\mathcal H}}
\nc{\cali}{{\mathcal I}}
\nc{\calj}{{\mathcal J}}
\nc{\call}{{\mathcal L}}
\nc{\calm}{{\mathcal M}}
\nc{\caln}{{\mathcal N}}
\nc{\calo}{{\mathcal O}}
\nc{\calp}{{\mathcal P}}
\nc{\calr}{{\mathcal R}}
\nc{\calt}{{\mathcal T}}
\nc{\calw}{{\mathcal W}}
\nc{\calx}{{\mathcal X}}
\nc{\CA}{\mathcal{A}}

\nc{\fraka}{{\mathfrak a}}
\nc{\frakB}{{\mathfrak B}}
\nc{\frakb}{{\mathfrak b}}
\nc{\frakd}{{\mathfrak d}}
\nc{\frakF}{{\mathfrak F}}
\nc{\frakf}{{\mathfrak f}}
\nc{\frakg}{{\mathfrak g}}
\nc{\frakL}{{\mathfrak L}}
\nc{\frakm}{{\mathfrak m}}
\nc{\frakM}{{\mathfrak M}}
\nc{\frakMo}{{\mathfrak M}^0}
\nc{\frakp}{{\mathfrak p}}
\nc{\frakw}{{\mathfrak w}}
\nc{\frakx}{{\mathfrak x}}
\nc{\ox}{\overline{\frakx}}
\nc{\frakX}{{\mathfrak X}}
\nc{\fraky}{{\mathfrak y}}

\nc{\li}[1]{\textcolor{blue}{Li: #1}}
\nc{\wys}[1]{\textcolor{green}{William: #1}}
\nc{\byhs}[1]{\textcolor{red}{Bingyong: #1}}

\nc{\rrb}[1]{[#1]}
\nc{\rrrb}[1]{\{#1\}}
\nc{\ideal}[1]{\langle #1\rangle}
\nc{\refl}[1]{\overline{#1}}
\nc{\rrB}{{reflexive Rota-Baxter}\xspace}
\nc{\rrrB}{{radical reflexive Rota-Baxter}\xspace}
\nc{\srb}{{strict Rota-Baxter}\xspace}
\nc{\Srb}{{Strict Rota-Baxter}\xspace}

\renewcommand\geq{\geqslant}
\renewcommand\leq{\leqslant}

\nc\rbop{{\lc\,\,\rc}}
\nc\rbopi[1]{{\lc_{#1} \, \rc_{#1}}}

\nc{\redtext}[1]{\textcolor{red}{#1}}

\begin{document}

\title[Structures of mixable shuffle algebras and Rota-Baxter algebras]
{Structure theorems of mixable shuffle algebras and free commutative Rota-Baxter algebras}
\author{Li Guo}
\address{Department of Mathematics and Computer Science,
         Rutgers University,
         Newark, NJ 07102}
\email{liguo@newark.rutgers.edu}
\author{Bingyong Xie}
\address{Department of Mathematics, Peking University, Beijing, China}
\email{byhsie@math.pku.edu.cn}


\begin{abstract}
We study the ring theoretical structures of mixable shuffle algebras and their associated free commutative Rota-Baxter algebras. For this study we utilize the connection of the mixable shuffle algebras with the overlapping shuffle algebra of Hazewinkel, quasi-shuffle algebras of Hoffman and quasi-symmetric functions. This connection allows us to apply methods and results on shuffle products and Lyndon words on ordered sets.
As a result, we obtain structure theorems for a large class of mixable shuffle algebras and free commutative Rota-Baxter algebras with various coefficient rings.
\end{abstract}

\maketitle


\setcounter{section}{0}

\section{Introduction}
In this paper, all rings and algebras are assumed to be unitary unless otherwise specified. Let $\bfk$ denote a commutative ring. By an algebra we mean a $\bfk$-algebra and by a tensor product we mean the tensor product over $\bfk$.

\subsection{Rota-Baxter algebras and mixable shuffle algebras}
Given a commutative ring $\bfk$ and a $\lambda\in \bfk$, a {\bf Rota-Baxter algebra of weight $\lambda$} is an associative $\bfk$-algebra $R$ together with a $\bfk$-linear operator $P$ on $R$ such that
\begin{equation}
P(x)P(y)=P(xP(y))+P(P(x)y)+ \lambda P(xy), \forall x,y\in R.
\mlabel{eq:RB}
\end{equation}
Such an operator is called a {\bf Rota-Baxter operator} (of weight $\lambda$). This operator is an abstraction of the integration operator $P(f)(x):=\int_0^x f(t)\,dt$ where the above identity is simply the integration by parts formula. This operator also include as special cases numerous other operators in mathematics and physics, such as the summation operator of functions, partial sum operator for sequences and projection operator on Laurent series, as well as the operator on distributions in the paper~\mcite{Ba} where G. Baxter first defined this operator. Such broad connections lead to many applications of Rota-Baxter algebras~\mcite{Ag,AGKO,Bai,C-K1,E-G,E-G-K3,EGM,Guum,Gust,G-Z,Ro2} which further motivate the theoretical study of Rota-Baxter algebras.
See the introductory and survey articles~\mcite{EGsu,Gda,Gsu,Ro2} for further details.

As a first step in their theoretical study, free commutative Rota-Baxter algebras were constructed by Cartier and Rota~\mcite{Ca,Ro1} with certain restrictions. A general construction was obtained by one of the authors and Keigher~\mcite{G-K1,G-K2} in terms of mixable shuffle products.
For a commutative $\bfk$-algebra $A$, let $\sha_{\bfk,\lambda}(A)$ be the free commutative Rota-Baxter algebra of weight $\lambda$ generated by $A$. It is shown in~\mcite{G-K1} that
\begin{equation}
\sha_{\bfk,\lambda}(A) = A\ot \sh_{\bfk,\lambda}(A)
\mlabel{eq:decomp}
\end{equation}
where $\sh_{\bfk,\lambda}(A)$ (denoted by $\sha^+_{\bfk,\lambda}(A)$ in~\mcite{G-K1,G-K2}) is {\bf the mixable shuffle algebra of weight $\lambda$} generated by $A$. The precise definitions will be recalled in Section~\mref{ss:msh}. Thus the study of free commutative Rota-Baxter algebras is reduced to the study of mixable shuffle algebras.

\subsection{Overlapping shuffle algebra and quasi-symmetric functions}
During the same period of time when mixable shuffle product was constructed, Hazewinkel~\mcite{Ha0,Ha} defined the overlapping shuffle algebra and showed that it gives another description of the algebra of quasi-symmetric functions. He then used the language and methods on Lyndon words of shuffles algebras to extend the well-known theorem of Radford~\mcite{Ra} that the shuffle algebra with rational coefficients is a polynomial algebra generated by the set of Lyndon words to the algebra of quasi-symmetric functions with rational coefficients.
More generally, Hoffman~\mcite{Ho} showed that his quasi-shuffle algebras, also introduced during the same period of time, are polynomial algebras on Lyndon words when rational coefficients are considered.

The theory of these algebras with integer coefficients developed more slowly. As commented in~\mcite{Ha,Ha2}, Ditters announced in his 1972  paper~\mcite{Di} that the algebra of quasi-symmetric functions with integer coefficients is a polynomial algebra. But there was a gap in his proof, as well as in the quite a few subsequent efforts to prove the statement. Eventually, Hazewinkel was able to provide a correct proof (Theorem~\mref{thm:qlyn}.(\mref{it:dh})). So we will call this statement the {\bf Ditters Conjecture} or the {\bf Ditters-Hazewinkel Theorem}.

\subsection{Mixable shuffles and overlapping shuffles}
As we will see later in Section~\mref{ss:moq}, the overlapping shuffle algebra, generalized overlapping shuffle algebras and quasi-shuffle algebras are all special cases of mixable shuffle algebras.
In this paper we extend the results and methods for these special cases, especially from~\mcite{Ha}, to study more general mixable shuffle algebras with various coefficient rings. We then study the ring theoretical structure of free commutative Rota-Baxter algebras through the tensor decomposition in Eq.~(\mref{eq:decomp}).
This paper can be regarded as a continuation of our earlier studies~\mcite{EGsh,Gudom,G-K1,G-K2} on this subject.

In analogy to the cases of the overlapping shuffle algebra and quasi-symmetric functions, the structure of a mixable shuffle algebra depends on its base ring $\bfk$, as well as its weight $\lambda$, especially for those mixable shuffle algebras that appear in the construction of free commutative Rota-Baxter algebras. So we will consider mixable shuffle algebras and Rota-Baxter algebras in these separate cases. For notational simplicity, we will take the base ring $\bfk$ to be $\QQ$, $\FF_p$, $\ZZ_p$ or $\ZZ$. See Table~\mref{tb:sum} for a summary of previous and new results.



When $\bfk=\QQ$, Radford's theorem and its generalizations by Hazewinkel~\mcite{Ha} and Hoffman~\mcite{Ho} can be quite easily generalized further to mixable shuffle algebras (Theorem~\mref{thm:msq}) and then to free commutative Rota-Baxter algebras (Theorem~\mref{thm:rbl}). This is presented in Section~\mref{sec:q} after preliminary notations and results.

The situation is already quite different in the case of $\bfk=\FF_p$ which is considered in Section~\mref{sec:chp}. By a careful study of the Lyndon words, we obtain the structure theorem (Theorem~\mref{thm:pmsh}) for a quite large class of mixable shuffle algebras. This leads to the structure theorem of a quite large class of free commutative Rota-Baxter algebras (Theorem~\mref{thm:rbafp}), including those generated by a finite set.

In Section~\mref{sec:zp}, we lift the results in Section~\mref{sec:chp} from $\FF_p$ to $\ZZ_p$ by studying the reduction map $\ZZ_p\to \FF_p$. As is often the case in this lifting process, we can only recover part of the information and obtain a less precise structure theorem on the mixable shuffle algebras with $\ZZ_p$-coefficients (Theorem~\mref{thm:isomor}), which translates to a less precise structure theorem on the free commutative Rota-Baxter algebras with $\ZZ_p$-coefficients (Theorem~\mref{thm:rbazp}). Nevertheless, in the case that we are most interested in and includes the overlapping shuffle algebra, we show that the mixable shuffle algebra is a polynomial algebra generated by an explicitly defined set.

In the final Section~\mref{sec:int}, we give a local-global principle extracted from Hazewinkel's elegant proof of the Ditters-Hazewinkel Theorem~\mcite{Ha} mentioned above.
This principle allows us to ``glue"  together our local results over $\QQ$ and $\ZZ_p$, for all $p$, to obtain results over $\ZZ$.
As a result, we generalize the Ditters-Hazewinkel Theorem from the mixable shuffle algebra on free abelian semigroup with one generator to those with countably many generators (Theorem~\mref{thm:intfr2}).
We obtain a similar polynomial algebra in free commutative Rota-Baxter algebra generated by a set (Theorem~\mref{thm:rbaz}).

\medskip

\noindent
{\bf Acknowledgements: } Both authors thank the Max Planck Institute for Mathematics at Bonn where this research was carried out. The first author acknowledges support from NSF grant DMS-0505643.

\section{Structure theorems on $\QQ$}
\mlabel{sec:q}
In this section we first review the construction of free commutative Rota-Baxter algebras in terms of mixable shuffle algebras obtained in~\mcite{G-K1,G-K2}. We then relate mixable shuffle algebras to the overlapping shuffle algebra and generalized overlapping shuffle algebras of Hazewinkel~\mcite{Ha,Ha2}, and quasi-shuffle algebras of Hoffman~\mcite{Ho}. This connection allows us to extend the study of overlapping shuffle algebra and quasi-shuffle algebras to the study the structure of mixable shuffle algebras and free commutative Rota-Baxter algebras with base ring $\QQ$. This connection will also be used in later sections for other base rings.

\subsection{Mixable shuffle algebras and free commutative Rota-Baxter algebras}
\mlabel{ss:msh}
We briefly recall the construction of mixable shuffle algebras and free commutative Rota-Baxter algebras~\mcite{G-K1,G-K2}.

Let $A$ be a commutative $\bfk$-algebra {\em that is not necessarily unitary}. For a given $\lambda\in \bfk$, the {\bf mixable shuffle algebra of weight $\lambda$ generated by $A$} (with coefficients in $\bfk$) is the $\bfk$-module
\begin{equation}
\sh(A):= \sh_{\bfk,\lambda} (A)= \bigoplus_{k\ge 0}
    A^{\otimes k}
= \bfk \oplus A\oplus A^{\otimes 2}\oplus \cdots
\mlabel{eq:mshde}
\end{equation}
equipped with the {\bf
mixable shuffle product $\shprl$ of weight $\lambda$} defined as
follows.

For pure tensors
$\fraka=a_1\ot \ldots \ot a_m\in A^{\ot m}$ and $\frakb=b_1\ot
\ldots \ot b_n\in A^{\ot n}$, a {\bf shuffle} of $\fraka$ and
$\frakb$ is a tensor list of $a_i$ and $b_j$ without change the
natural orders of the $a_i$s and the $b_j$s. More generally, for the
fixed $\lambda\in \bfk$, a {\bf mixable shuffle} (of weight
$\lambda$) of $\fraka$ and $\frakb$ is a shuffle of $\fraka$ and
$\frakb$ in which some (or {\it none}) of the pairs $a_i\ot b_j$ are
merged into $\lambda\, a_i b_j$. Then define
\begin{equation}
\fraka\shpr \frakb=\fraka \shprl \frakb= \sum {\rm\ mixable\ shuffles\ of\ }\fraka {\rm\ and\ } \frakb
\mlabel{eq:msh}
\end{equation}
where the subscript $\lambda$ is often suppressed when there is no danger of confusion.
For example,
\begin{eqnarray*}
a_1 \shpr (b_1\ot b_2):&=& a_1 \shprl (b_1\ot b_2)\\
&=& \underbrace{a_1\ot b_1\ot b_2 + b_1\ot a_1\ot b_2
+ b_1\ot b_2\ot a_1}_{\rm shuffles}
+ \underbrace{\lambda (a_1 b_1)\ot b_2 + \lambda b_1\ot (a_1 b_2)}_{\rm merged\ shuffles}.
\end{eqnarray*}
With $\bfone\in \bfk$ as the unit, this product makes $\sh_{\bfk,\lambda}(A)$ into a commutative $\bfk$-algebra.
See~\mcite{G-K1} for further details of the mixable shuffle product.
When $\lambda=0$, we simply have the shuffle product which is also defined when $A$ is only a $\bfk$-module, treated as an algebra with zero multiplication.

The product $\shprl$ can also be defined by the following recursion~\mcite{EGsh,G-Z2} which gives the connection with quasi-shuffle algebras of Hoffman~\mcite{Ho}.
First define the multiplication by $A^{\ot 0}=\bfk$ to be the scalar product. In particular, $\bfone$ is the identity.
For any $m,n\geq 1$ and
$\fraka:=a_1\ot\cdots \ot a_m\in A^{\ot m}$, $\frakb:=b_1\ot \cdots\ot b_n\in A^{\ot n}$, define
$a \shprl b$ by induction on the sum $m+n$. Then $m+n\geq 2$. When $m+n=2$, we have
$a=a_1$ and $b=b_1$. Define
\begin{equation} a\shprl b = a_1\ot b_1 + b_1\ot a_1 + \lambda a_1b_1.
\label{eq:quasi0}
\end{equation}
Assume that $\fraka\shprl \frakb$ has been defined for $m+n\geq k\geq 2$ and consider $\fraka$ and $\frakb$
with $m+n=k+1$. Then $m+n\geq 3$ and so at least one of $m$ and $n$ is greater than 1.
Then we define
\begin{equation}
  \fraka \shprl \frakb =\left\{ \begin{array}{l}
  a_1\ot  b_1\ot  \cdots \ot b_n  + b_1\ot \big(a_1\shprl (b_2\ot \cdots\ot b_n)\big) \\
    \qquad \qquad + \lambda(a_1b_1)\ot  b_2\ot \cdots\ot  b_n, {\rm\ when\ } m=1, n\geq 2, \\
    a_1 \ot\big ((a_2\ot \cdots\ot  a_m)\shprl b_1 \big) + b_1\ot a_1\ot \cdots\ot a_m \\
 \qquad \qquad + \lambda(a_1b_1) \ot  a_2\ot \cdots\ot  a_m,     {\rm\ when\ } m\geq 2, n=1, \\
     a_1\ot  \big ((a_2\ot \cdots\ot a_m)\shprl
(b_1\ot \cdots\ot  b_n)\big ) + b_1\ot  \big ((a_1\ot  \cdots \ot a_m)\shprl (b_2 \ot \cdots \ot  b_n)\big) \notag \\
 \qquad \qquad  + \lambda(a_1 b_1)  \big ( (a_2\ot \cdots\ot a_m) \shprl
     (b_2\ot \cdots\ot  b_n)\big ),
     {\rm\ when\ } m, n\geq 2.
     \end{array} \right .
\mlabel{eq:quasi}
\end{equation}
Here the products by $\shprl$ on the right hand side of the equation are well-defined
by the induction hypothesis.

\smallskip

Now let $A$ be a (unitary) $\bfk$-algebra. We define the tensor product algebra
\begin{equation} \sha(A):=\sha_{\bfk,\lambda}(A)= A\ot \sh_{\bfk,\lambda}(A) =A \oplus A^{\ot 2} \oplus \cdots.
\mlabel{eq:freerb}
\end{equation}
Define a $\bfk$-linear operator $P_A$ on
$\sha (A)$ by assigning
\[ P_A( x_0\otimes x_1\otimes \cdots \otimes x_n)
=\bfone_A\otimes x_0\otimes x_1\otimes \cdots\otimes x_n, \]
for all
$x_0\otimes x_1\otimes \cdots\otimes x_n\in A^{\otimes (n+1)}$
and extending by additivity.
Let $j_A:A\rar \sha (A)$ be the canonical inclusion map.

\begin{theorem} {\bf \mcite{G-K1}}
\begin{enumerate}
\item
The pair $(\sha (A),P_A)$, together with the natural embedding
$j_A:A\rightarrow \sha (A)$, is a free commutative Rota-Baxter
$\bfk$-algebra of weight $\lambda$ on $A$. In other words, for any
Rota-Baxter $\bfk$-algebra $(R,P)$ and any $\bfk$-algebra
homomorphism $\varphi:A\rar R$, there exists a unique Rota-Baxter
$\bfk$-algebra homomorphism $\tilde{\varphi}:(\sha (A),P_A)\rar
(R,P)$ such that $\varphi = \tilde{\varphi} \circ j_A$ as
$\bfk$-algebra homomorphisms. \mlabel{it:freea}
\item
When $X$ is a set. The pair $(\sha(\bfk[X]), P_{\bfk[X]})$, together with the natural embedding $j_X: X\rightarrow \bfk[X] \rightarrow \sha(\bfk[X])$, is a free commutative Rota-Baxter $\bfk$-algebra on the set $X$ of weight $\lambda$.
\mlabel{it:freex}
\end{enumerate}
\mlabel{thm:shua}
\end{theorem}

\subsection{Mixable shuffles, overlapping shuffles and quasi-shuffles}
\mlabel{ss:moq}
Let $S$ be a semigroup and let $\bfk\,S=\sum_{s\in S} \bfk\,s$ be the semigroup nonunitary $\bfk$-algebra. Then a canonical $\bfk$-basis of $(\bfk\,S)^{\ot k}, k\geq 0$, is the set $S^{\ot k}:=\{s_1\ot \cdots \ot s_k\ |\  s_i\in S, 1\leq i\leq k\}$. Thus a canonical $\bfk$-basis of $\sh_{\bfk,\lambda}(A)$ is
\begin{equation}
M\tcon(S):=\{1\}\cup \{ u_1\ot \cdots \ot u_r\ |\ u_i\in S, 1\leq i\leq r, r\geq 1\}.
\mlabel{eq:tmon}
\end{equation}
With the tensor concatenation, $M\tcon(S)$ is simply the free monoid generated by $S$.
We use the tensor concatenation instead of the usual concatenation for the product since we need to use the concatenation to denote the product in $S$ when $S$ is a semigroup.
Elements in $M\tcon(S)$ are still called {\bf words} from the set $S$. Then we have
$$ \sh_{\bfk,\lambda}(S):=\sh_{\bfk,\lambda}(\bfk S)=\bfk\, M\tcon(S).$$
We denote $\sh_{\bfk,\lambda}(S)$ for $\sh_{\bfk,\lambda}(\bfk\, S)$ to make clear the connection with $S$ and to simplify the notation.

Let $S$ be a monoid and let $\bfk\,S$ be the (unitary) $\bfk$-algebra. As in Eq.~(\mref{eq:decomp}) we have the free commutative Rota-Baxter algebra
\begin{equation}
 \sha_{\bfk,\lambda} (\bfk S) =(\bfk S)\ot \sh_{\bfk,\lambda}(S).
 \mlabel{eq:factor2}
 \end{equation}
It is in fact the free commutative Rota-Baxter algebra generated by the monoid $S$ in the sense that it comes from the left adjoint functor of the forgetful functor from the category of commutative Rota-Baxter algebras to the category of commutative multiplicative monoids.

Now let $S$ be the multiplicative semigroup $\{x^i\}_{i\geq 1}$.
Then
$$ M\tcon(S)= \{ x^{a_1}\ot \cdots \ot x^{a_k}\ |\ a_j\geq 1, 1\leq j\leq k, k\geq 0\}.$$
It is in bijection with the set of vectors
$$ \{ [a_1,\cdots, a_k]\ |\ a_j\geq 1, 1\leq j\leq k, k\geq 0\}$$
and with the set of polynomials
$$ \left \{ \sum_{1\leq i_1<\cdots <i_n} X_{i_1}^{a_1}\cdots X_{i_k}^{a_k}\ |\
    a_j\geq 1, 1\leq j \leq k, k\geq 0 \right\} \subseteq \bfk[X_i, i\geq 1].
$$
Through the first bijection, we obtain the isomorphism of $\sh_{\bfk,1}(S)$ with the {\bf overlapping shuffle algebra}
$$ \bfk \{ [a_1,\cdots, a_k]\ |\ a_j\geq 1, 1\leq j\leq k, k\geq 0\}$$
defined by Hazewinkel~\mcite{Ha0}. See~\mcite{Ha0} for more details and a more precise definition of the product in terms of order preserving injective maps (see also~\mcite{Ca} and~\mcite{Eh}).
Through the second bijection, we obtain the isomorphism of $\sh_{\bfk,1}(S)$ with the algebra $QSym_\bfk (S)$ of quasi-symmetric functions~\mcite{Ge}.

Let $S$ be a graded semigroup $S=\coprod_{i\geq 0} S_i$, $S_iS_j\subseteq S_{i+j}$ such that $|S_i|<\infty$, $i\geq 0$.
Then with $\lambda =1$, the mixable shuffle algebra $\sh_{\lambda}(S)$ is isomorphic to the {\bf quasi-shuffle algebra} defined by Hoffman~\mcite{Ho,EGsh,G-Z2}.

For a general semigroup $S$, the mixable shuffle algebra $\sh_{\bfk,1}(S)$ of weight $1$ coincides with the {\bf generalized overlapping shuffle algebra} on $S$~\mcite{Ha2}.

Let $(S,<)$ be an ordered set. Extend the order on $S$ to the
{\bf lexicographic order} $<\lexord$ on $M\tcon(S)$. Thus, for $u,v\in M\tcon(S)$, $u<\lexord v$ if and
only if either $v=u\ot x$ for some non-empty word $x$, or $u=x\ot a\ot u',
v=x\ot b\ot v'$ for some words $x,u',v'$ and some letter $a,b$ with $a<b$.
Recall that a {\bf Lyndon word} in $M\tcon(S)$ is a non-empty word $w$ such that
if $w=u\ot v$ with $u,v\neq 1$, then $w<\lexord v$. Let $\Lyn=\Lyn(S)$ be
the set of Lyndon words in $M\tcon(S)$.

The following theorem summarizes what is known about when a mixable shuffle algebra is a polynomial algebra.
\begin{theorem}
\begin{enumerate}
\item {\bf (\cite{Ra}\cite[Theorem 6.1]{Re})}
Let $S$ be an ordered set. Then $\sh_{\QQ,0}(S)$, namely the shuffle algebra $Sh(S)$ on $S$ with coefficients in $\QQ$, is isomorphic to $\QQ[\Lyn(S)]$.
\mlabel{it:rad}
\item
{\bf (Hazewinkel-Hoffman Theorem~\cite{Ha},\cite[Theorem 2.6.]{Ho})} Let $S$ be an ordered abelian semigroup. Then $\sh_{\QQ,1}(S)$, namely the quasi-shuffle algebra on $S$ with coefficients in $\QQ$, is isomorphic to $\QQ[\Lyn(S)]$.
\mlabel{it:hh}
\item
{\bf (Ditters-Hazewinkel Theorem~\mcite{Di,Ha})}
Let $S$ be the free abelian semigroup with one generator. Then $\sh_{\ZZ,1}(S)$, namely the $\ZZ$-algebra of overlapping shuffles, and the algebra quasi-symmetric functions with integer coefficients, is a polynomial algebra.
\mlabel{it:dh}
\end{enumerate}
\mlabel{thm:qlyn}
\end{theorem}
Thus quite much is known about the mixable shuffle algebras with coefficients in $\QQ$ and with weight $0$ or $1$, but little is known in the other cases. One of our main goals in this paper is to extend this theorem to the cases for other coefficient rings and other weights, as summarized in Table~\mref{tb:sum}.
\bigskip

\begin{table}
\centering
\caption{Structure of $\sh_{\bfk,\lambda}(S)$}
\vspace{0.05 in}
\begin{tabular}{||c|c|c|c|c||}
\hline
& base ring $\bfk$ & weight $\lambda$ & \begin{tabular}{c} ordered set or\\ semigroup  $S$\end{tabular} & reference \\
\hline \hline
Radford~\cite{Ra} & $\QQ$ & 0 & ordered set& Theorem~\mref{thm:qlyn}.(\mref{it:rad})\\
\hline
Hoffman~\cite{Ho} & $\QQ$ & 1 & \begin{tabular}{c} ordered abelian\\
     semigroups \end{tabular}
     & Theorem~\mref{thm:qlyn}.(\mref{it:hh})\\
\hline
Hazewinkel~\cite{Ha} & $\QQ, \ZZ_p,\ZZ$ & 1 & $\ZZ_{>0}$ & Theorem~\mref{thm:qlyn}.(\mref{it:hh}) \& (\mref{it:dh}) \\
\hline
\multirow{9}{*}{This paper} &
$\QQ$ & $\neq 0$ & \begin{tabular}{c} ordered abelian \\ semigroups \end{tabular} &
Theorem~\mref{thm:msq} \\ &&&& \\
 &    $\FF_p$ & 0 & ordered set & Theorem~\mref{thm:psh} \\
    &&&& \\
  &   $\FF_p$ & $\neq 0$ & $S\in \pG, \jG$ & Theorem~\mref{thm:pmsh} \\
    &&&&\\
   &  $\ZZ_p$ & $p$-unit & $S\in \fG, \jG$ & Theorem~\mref{thm:isomor} \\
    &&&&\\
    & $\ZZ$ & $\pm 1$ & $S\cong \ZZ_{>0}^{n}$ or $\ZZ_{>0}^{(\infty)}$ & Theorem~\mref{thm:intfr} \& \mref{thm:intfr2}
\\ \hline
\end{tabular}
\mlabel{tb:sum}
\end{table}
\bigskip

We first consider the easy case when $\bfk=\QQ$ and $\lambda\in \QQ$ is arbitrary.
\begin{theorem}
Let $S$ be an ordered abelian semigroup and let $\lambda$ be in $\QQ$. Then $\sh_{\QQ,\lambda}(S)$ is isomorphic to $\QQ[\Lyn(S)]$.
\mlabel{thm:msq}
\end{theorem}
\begin{proof}
Fix a $\lambda\in \QQ$. If $\lambda =0$, then by definition, $\sh_{\QQ,\lambda}(\sg)$ is the shuffle algebra $Sh(\sg)$ on the $\QQ$-vector space $\QQ\, \sg$. By Theorem~\mref{thm:qlyn}.(\mref{it:rad}), we have $\sh_{\QQ,0}(\sg)=\QQ[\Lyn]$.
If $\lambda=1$, then as was shown in~\mcite{EGsh} and \mcite{G-Z2}, $\sh_{\QQ,1}(\sg)$ is the quasi-shuffle $\QQ$-algebra on the semigroup $S$ and thus is $\QQ[\Lyn(S)]$ by Theorem~\mref{thm:qlyn}.(\mref{it:hh}).

If $\lambda\neq 0,1$, the algebra isomorphism
$$
\begin{aligned}
f: &\sha_{\QQ,\lambda}(\QQ S) \to \sha_{\QQ,1}(\QQ S), \\
&a_0\ot \cdots \ot a_n\mapsto \lambda^n (a_0\ot \cdots\ot a_n), \forall a_0\ot \cdots \ot a_n \in \QQ S^{\ot (n+1)}
\end{aligned}
$$
{}from \cite{EGsh} (Lemma 2.8 and the comments afterward) restricts to an algebra isomorphism
$$
\begin{aligned}
f: &\sh_{\QQ,\lambda}(\QQ S) \to \sh_{\QQ,1}(\QQ S), \\
&a_1\ot \cdots \ot a_n\mapsto \lambda^n (a_1\ot \cdots\ot a_n), \forall a_1\ot \cdots \ot a_n \in \QQ S^{\ot n}.
\end{aligned}
$$
Thus a Lyndon word $\omega \in
\sh_{\QQ,1}(S)$ is sent to $\lambda^{\leng(\omega)} \omega\in
\sh_{\QQ,\lambda}(S)$ where $\leng(\omega)$ is the length of the
\word $\omega$. Since $\lambda\in \QQ$ is invertible,
$\sh_{\QQ,\lambda}(S)$ is still generated by $\Lyn(S)$. Thus
the theorem holds for all $\lambda\in \QQ$.
\end{proof}

\subsection{Free commutative Rota-Baxter algebras over a $\QQ$-algebra}
\mlabel{ss:rbaq}

We now apply Theorem~\mref{thm:msq} to free commutative Rota-Baxter algebras.

\begin{theorem}
Let $\sg$ be an ordered abelian monoid and let $\QQ \sg$ be the monoid algebra. Then
\begin{equation}
\sha_{\QQ,\lambda}(\QQ S)=\QQ S \ot \QQ[\Lyn(\sg)],
\mlabel{eq:shq}
\end{equation}
where $\Lyn(\sg)$ is the set of Lyndon words on $\sg$.
In particular, let $X$ be an ordered set. Let $M\com(X)$ be the free abelian monoid generated by $X$. Then
\begin{equation}
\sha_{\QQ,\lambda}(\QQ[X])=\QQ[\overline{\Lyn}(M\com(X))],
\mlabel{eq:shqx}
\end{equation}
where
$$ \overline{\Lyn}(M\com(X)):=X\cup \{1\ot w\ |\ w \in \Lyn(M\com(X))\}.$$
\mlabel{thm:rbl}
\end{theorem}

\begin{proof}
By Theorem~\mref{thm:msq} and Eq.~(\mref{eq:factor2}), we have
$\sha_{\QQ,\lambda}(\QQ S)=\QQ S \ot \QQ[\Lyn]$
by Eq.~(\mref{eq:decomp}).

For the second statement, let $X$ be an ordered set.
Then $\QQ[X]=\QQ  M\com(X)$ and
$$ \sha_{\QQ,\lambda}(\QQ[X])=\QQ[X]\ot \sh_{\QQ,\lambda}(M\com(X))
    =\QQ[X]\ot \QQ[\Lyn(M\com(X))]=\QQ[\overline{\Lyn}(M\com(X))].$$
\end{proof}

\section{Structure theorems on $\FF_p$}
\mlabel{sec:chp}
Given a prime number $p$, we now consider the algebra structure of the mixable shuffle algebras $\sh_{\FF_p,\lambda}(S)$ where $S$ is an ordered semigroup with base ring $\FF_p$. Here the situation is quite different from the case when the base ring is $\QQ$. As an easy
illustration, let $x\in S$, then the shuffle product $x^{\ssha p}=x^{\shpr_0 p}=p! x^{\ot p}=0$ in
$\sh_{\FF_p,0}(X)$. We will show that this phenomenon prevails when
the weight $\lambda$ is zero and, as a result, $\sh_{\FF_p,0}(S)$ has
no polynomial subalgebras. When $\lambda \neq 0$, the structure of $\sh_{\FF_p,\lambda}(S)$ is more diversified. For a large class of abelian semigroups $S$, including free semigroups, free monoids, $p$-nilpotent groups and $p$-idempotent groups, we determine the factorization of $\sh_{\FF_p,\lambda}(S)$ into a polynomial part and a non-polynomial part. We then apply these structure theorems to the free commutative Rota-Baxter algebras $\sha_{\FF_p,\lambda}(\FF_p S)$ with coefficients in $\FF_p$.

\subsection{Notations and preparatories}

Let $(S,<)$ be an ordered set and let the free monoid $M\tcon(S)$ be as defined in Eq.~(\mref{eq:tmon}).
Recall that we use $<\lexord$ to denote the lexicographic order on $M\tcon(S)$ induced from the order on $S$.
We will use another order $<\lengord$ on $M\tcon(S)$.
\begin{defn}
Let $(S,<)$ be an ordered semigroup. For $u=u_1\ot \cdots \ot u_r\in S^{\ot r}$ and $v=v_1\ot \cdots \ot v_s\in S^{\ot s}$, define
\begin{equation}
u<\lengord v \Leftrightarrow
\left\{\begin{array}{l} r<s {\rm\ or\ } \\
     r=s {\rm\ and\ \exists\ } 1\leq i\leq r, {\rm\ such\ that\ } u_1=v_1,\cdots, u_{i-1}=v_{i-1},u_i<v_i.
\end{array} \right .
\mlabel{eq:order}
\end{equation}
$<\lengord$ will be called the {\bf pro-length order} (or \mbox{\bf L-order} for short).
\mlabel{de:order}
\end{defn}
We note that, when $u$ and $v$ have the same length, $u<\lexord v$ if and only if $u<\lengord v$.
Recall that a well-ordered set is a totally ordered set whose every non-empty subset has a smallest element.

\begin{lemma}
Let $(S,<)$ be a well-ordered set. Then the \lord $<\lengord$ defines a well order on the set $M\tcon(S)$.
\mlabel{lem:order}
\end{lemma}
\begin{proof}
$<\lengord$ is clearly a total order on $M\tcon(S)$. Let $T$ be a
non-empty subset of $M\tcon(S)$. Define $T_0$ to be the subset of
$T$ consisting of \words of the smallest length $r$, $T_1$ to
be the subset of $T_0$ consisting of tensors $u_1\ot \cdots \ot u_r$
such that $u_1$ is the smallest, $T_2$ to be the the subset of $T_1$
consisting of tensors $u_1\ot \cdots \ot u_r$ such that $u_2$ is the
smallest, $\cdots$, $T_r$ to be the subset of $T_{r-1}$ consisting of
tensors $u_1\ot \cdots \ot u_r$ such that $u_r$ is the smallest.
Then the smallest element of $T$ is the unique element of $T_r$.
\end{proof}

We list the following results for later references.
\begin{theorem}
\begin{enumerate}
\item
{\bf (Chen-Fox-Lyndon factorization)~\mcite{Re}}
Any word $w\in M\tcon(S)$ can be written uniquely as a tensor product of Lyndon words
$$w=w_1^{\otimes i_1}\otimes\cdots \otimes w_{k}^{\otimes
i_k},\quad w_1>\cdots >w_k,\ i_1, \cdots , i_k\geq 1.$$
\mlabel{it:cfl}
\item
{\bf (Tensor form of freshman's dream) \cite[Theorem~4.1 ]{Gust}}
For any $w=w_1\ot \cdots \ot w_n\in M\tcon(S)$ and $\lambda\in \bfk$,
\begin{equation}
w^{\shprl p} \equiv \lambda^{(p-1)(n-1)} w_1^p\ot \cdots \ot w_n^p \mod p.\mlabel{eq:41}
\end{equation}
\mlabel{it:41}
\end{enumerate}
\mlabel{lem:known}
\end{theorem}

\noindent
{\bf Notation:}
For $u\in \sh_{\bfk,\lambda}(S)$ and $w\in M\tcon(S)$, we write
$$ u=w +\text{\ lower \lord terms}$$
if $u-w$ is a linear combination of words in $M\tcon(S)$ with \lord less than $w$.

\begin{lemma}
The following statements hold in $\sh_{\ZZ,\lambda}(S)$.
\begin{enumerate}
\item
Let $w=w_1^{\otimes i_1}\otimes\cdots \otimes w_{k}^{\otimes
i_k}$ be the Chen-Fox-Lyndon factorization.
We have
$$w_1^{\shprl  i_1}\shprl \cdots \shprl  w_k^{\shprl  i_k}=(i_1!\cdots i_k!)w+\text{lower \lord terms}.
$$
\mlabel{it:basiclyndon2}
\item
Let $u$ be a Lyndon word and let $v$ be a word with $u>v$. Then
$$
u^{\ot s} \shprl  v= u^{\ot s}\ot v + \text{lower \lord terms}.
$$
\mlabel{it:lead1}
\item
Let $u$ be a Lyndon word and let $n_1,\cdots n_k$ be positive integers. Then
$$ u^{\otimes n_1} \shpr \cdots \shpr  u^{\otimes n_k}=
\frac{(n_1+\cdots +n_k)!}{n_1!\cdots n_k!}  u^{\otimes
(n_1+\cdots n_k)}+ \text{lower \lord terms}.$$
\mlabel{it:lead2}
\item
For any Lyndon word $u$ and integer $n=a_0+a_1p+a_2p^2+\cdots
a_kp^k$ with $a_0,\cdots, a_k\in \{ 0,1,\cdots, p-1\}$,
we have
\begin{equation}
(u^{\ot p^0})^{\shprl a_0}\shprl \cdots \shprl (u^{\ot p^{k}})^{\shprl a_k}= N_n u^{\ot n} + \text{ lower \lord terms},
\mlabel{eq:lead3}
\end{equation}
where $N_n$ is a $p$-adic unit.
\mlabel{it:lead3}
\end{enumerate}
\mlabel{lem:lead}
\end{lemma}

\begin{proof}
(\mref{it:basiclyndon2}).
As is well-known~\mcite{Re}, for the shuffle product $\ssha=\shpr_0$ (mixable shuffle product of weight 0), we have
$$w_1^{\shpr_0  i_1}\shpr_0 \cdots \shpr_0  w_k^{\shpr_0  i_k}=(i_1!\cdots i_k!)w+
\sum_{\leng(u)=\leng(w), u<w}\alpha_u u$$
for some natural integer $\alpha_u$.
By the definition of the mixable shuffle product of weight $\lambda$,
$$w_1^{\shprl  i_1}\shprl \cdots \shprl  w_k^{\shprl  i_k}=w_1^{\shpr_0  i_1}\shpr_0 \cdots \shpr_0  w_k^{\shpr_0  i_k}+ \text{\rm terms of length}<\leng(w).$$
Since either $\leng(u)=\leng(w)$ with $u<\lexord w$ or $\leng(u)<\leng(w)$ implies $u<\lengord w$, we are done.
\smallskip

(\mref{it:lead1}).
Let $v=v_1^{\otimes i_1}\ot\cdots\ot v_{k}^{\otimes i_k}$ be the Chen-Fox-Lyndon factorization. Since $v_1$ is a Lyndon word, we have $v>v_1$. Since it is assume that $v<u$,
we have $u>v_1$. Thus $u^{\ot s}\ot v= u^{\ot s}\ot v_1^{\otimes i_1}\ot\cdots\ot v_{k}^{\otimes i_k}$ is the Chen-Fox-Lyndon factorization of $u^{\ot s}\ot v$. Then by Item (\mref{it:basiclyndon2}),
$$ u^{\shprl s}\shprl v_1^{\shprl i_1}\shprl \cdots \shprl v_k^{\shprl i_k}=(s!)(i_1!)\cdots (i_k!)u^{\ot s}\ot v + \text{\rm lower \lord terms}. $$
On the other hand, applying Item (\mref{it:basiclyndon2}) separately to $u^{\ot s}$ and
$v=v_1^{\otimes i_1}\ot\cdots\ot v_{k}^{\otimes i_k}$, we have
$$ u^{\shprl s}\shprl v_1^{\shprl i_1}\shprl \cdots \shprl v_k^{\shprl i_k} = (s!)u^{\ot s}\shprl  ((i_1!)\cdots (i_k!))v + \text{terms with \lord lower than } u^{\ot s}\ot v.$$
This gives what we need.

(\mref{it:lead2}). By Item (\mref{it:basiclyndon2}) we have
$$\frac{1}{n_i!}u^{\shprl  n_i}=u^{\otimes n_i}+ \text{lower \lord
terms}.$$ So
$$ \begin{aligned}u^{\otimes n_1} \shprl  \cdots \shprl  u^{\otimes n_k} &=
\frac{1}{n_1!\cdots n_k!} u^{\shprl  (n_1+\cdots +n_k)}+ \text{terms with \lord lower than } u^{\otimes (n_1+\cdots +n_k)}\\
&= \frac{(n_1+\cdots +n_k)!}{n_1!\cdots n_k!}
u^{\otimes (n_1+\cdots +n_k)}+ \text{lower \lord terms},
\end{aligned}$$
as desired.

(\mref{it:lead3}) is a special case of (\mref{it:lead2}) since
$N_n=\frac{n!}{\prod_{j=0}^k (p^j!)^{a_j}}$ is a $p$-adic unit~\cite[Corollary 7.6]{Ha}.
\end{proof}

Let $A$ be a commutative $\bfk$-algebra. For a pure tensor $a$ in
$A^{\ot n}$, denote $a^{\ot k}$ to be the $k$ fold tensor power
of $a$. For a set $Y$ of pure tensors and a prime number $p$, denote
\begin{equation}
Y^{\ot k}=\{a^{\ot k}\ |\ a\in Y\},\quad  \rmt(Y)=\coprod_{k\geq 0} Y^{\ot p^k}.
\mlabel{eq:check}
\end{equation}
Here $\rmt$ stands for tensor power. When $Y=\Lyn$ is the set of Lyndon words in $\sh_{\bfk,\lambda}(S)$ where $S$ is an ordered semigroup, we denote $\TL=\rmt(\Lyn)$.

We will use the following proposition several times.
\begin{prop} Let $\bfk$ be either $\FF_p$ or $\ZZ_p$. Let $S$ be a well-ordered semigroup and let $\lambda\in \bfk$. Denote $\shpr=\shprl$.
\begin{enumerate}
\item
As a $\bfk$-algebra, $\sh_{\bfk,\lambda}(S)$ is generated by $\TL$ for any $\lambda\in \bfk$.
\mlabel{it:span}
\item
The subset
\begin{equation}
U:=\{1\}\cup \{w_1^{\shpr n_1}\shpr \cdots \shpr w_r^{\shpr n_r} |\ w_j\in \TL, w_1>\cdots>w_r, 1\leq n_j\leq p-1, 1\leq j\leq r, r\geq 1\}
\mlabel{eq:u}
\end{equation}
of $\sh_{\bfk,\lambda}(S)$ is linearly independent.
\mlabel{it:ind}
\end{enumerate}
\mlabel{pp:onto}
\end{prop}

\begin{proof}
(\mref{it:span}). Let $\sh_{\bfk,\lambda}(S)'$ be the $\bfk$-subalgebra of $\sh_{\bfk,\lambda}(S)$ generated by $\TL$. We just need to prove $M\tcon(S)\subseteq \sh_{\bfk,\lambda}(S)'$ by contradiction. First of all, the smallest element in $M\tcon(S)$ is the 1-tensor $s_0$ where $s_0$ denotes the smallest
element of the well-ordered semigroup $S$. Since $s_0$ is a Lyndon word, $s_0$ is in $\sh_{\bfk,\lambda}(S)'$. Therefore $M\tcon(S)\backslash \sh_{\bfk,\lambda}(S)'$ is not $M\tcon(S)$.
Suppose $M\tcon(S)\not\subseteq \sh_{\bfk,\lambda}(S)'$, then $M\tcon(S)\backslash \sh_{\bfk,\lambda}(S)'$ is not empty. Since by Lemma~\mref{lem:order}, $M\tcon(S)$ is a well-ordered set with respect to the \lord, there is a smallest element $w$ in $M\tcon(S)\backslash \sh_{\bfk,\lambda}(S)'$. Let $w=w_1^{\ot i_1}\ot \cdots \ot w_r^{i_r}, w_1>\cdots >w_r,$ be the Chen-Fox-Lyndon factorization of $w$.
\smallskip

Suppose $r=1$. Then $w=w_1^{\ot n}$ for some $n\geq 1$. Using the notation of Lemma~\mref{lem:lead}.(\mref{it:lead3}), we have
$$   w_1^{\ot n} = N_n^{-1} (w_1^{\ot p^0})^{\shpr a_0}\shpr \cdots \shpr (w_1^{\ot p^{r}})^{\shpr a_r} + \text{\rm terms with \lord lower than\ } w_1^{\ot n},$$
where $N_n$ is a $p$-adic unit. Since
$(w_1^{\ot p^0})^{\shpr a_0}\shpr \cdots \shpr (w_1^{\ot p^{r}})^{\shpr a_r}$
is a product of the elements $w_1^{\ot p^i},i\geq 0,$ that are already in $\TL$, this product is in $\sh_{\bfk,\lambda}(S)'$. By the minimality of $w=w_1^{\ot n}$, the other terms on the right hand side of the above equation are also in $\sh_{\bfk,\lambda}(S)'$. Thus
$w_1^{\ot n}$ is in $\sh_{\bfk,\lambda}(S)'$. This is a contradiction.
\medskip

Suppose $r>1$. Then by the Chen-Fox-Lyndon factorization, we have $w_2^{\ot i_2}\ot\cdots \ot w_r^{i_r}< w_1$. Hence Lemma~\mref{lem:lead} (\mref{it:lead1}) gives
$$w=w_1^{\ot i_1}\ot w_2^{\ot i_2}\ot\cdots \ot w_r^{i_r}=w_1^{\ot i_1} \shpr (w_2^{\ot i_2}\ot\cdots \ot w_r^{i_r}) + \text{\rm terms with \lord lower than } w.
$$
By the minimality of $w$, we have $w_1^{\ot s}, w_2^{\ot i_2}\ot\cdots \ot w_r^{i_r}\in \sh(S)'$ since they have lengths
shorter than $w$ and hence L-orders lower than $w$. Therefore, $w$ is also in $\sh(X)'$. This again is a contradiction and completes our proof that $M\tcon(S)\subseteq \sh_{\bfk,\lambda}(S)'$.
\medskip

(\mref{it:ind}).
Define
\begin{equation}
\Gamma=\{ \gamma:\TL \rightarrow \{ 0, \cdots, p-1 \}\ |\
\gamma \text{ has finite support} \}.
\mlabel{eq:gamma}
\end{equation}
Then we have
\begin{equation}
U=\{w_\gamma:=\bigshpr_{w\in \TL} w^{\shprl \gamma(w)}  \ |\ \gamma\in \Gamma\}.
\mlabel{eq:u3}
\end{equation}
For $\gamma\neq 0$, let the support of $\gamma$ be $\{w_1,\cdots,w_r\}\subseteq \TL$ with $w_1>\cdots>w_r$. Note that each $w_i$ is a $u^{\ot p^j}$ for some $u\in \Lyn$ and $j\geq 0$. Let $u_1>\cdots>u_t$ be such $u$'s in $\Lyn$. Then
$$ (w_1,\cdots,w_r)= (u_1^{\ot p^{i_{1,1}}},\cdots,u_1^{\ot p^{i_{1,a_1}}},u_2^{\ot p^{i_{2,1}}},\cdots, u_2^{\ot p^{i_{2,a_2}}},\cdots, u_t^{\ot p^{i_{t,1}}},\cdots,u_t^{\ot p^{i_{t,a_t}}}),$$
where $i_{j,1}>\cdots > i_{j,a_j}$, $a_j\geq 1, 1\leq j\leq t$.
Thus
\begin{equation}
\begin{aligned}
w_\gamma &=w_1^{\shprl \gamma(w_1)}\shprl\cdots\shprl w_r^{\shprl \gamma(w_r)}\\
&= \bigshprl_{j=1}^{t\ } \big(\bigshprl_{k=1}^{a_j}
(u_j^{\ot p^{i_{j,k}}})^{\shprl \gamma(u_j^{\ot p^{i_{j,k}}})}\big)
\\
&=\bigshprl_{j=1}^{t\ } \big(\bigshprl_{\ell=1}^{\infty\ }
(u_j^{\ot p^{\ell}})^{\shprl \gamma(u_j^{\ot p^{\ell}})}\big)
\end{aligned}
\mlabel{eq:uws}
\end{equation}
since $\gamma(u_j^{\ot p^\ell})=0$ outside the support of $\gamma$. Similarly,
\begin{equation}
\begin{aligned}
w_1^{\ot \gamma(w_1)}\ot\cdots\ot w_r^{\ot \gamma(w_r)}
&= \bigsot_{j=1}^t \big(\bigsot_{k=1}^{a_j}
(u_j^{\ot p^{i_{j,k}}})^{\ot \gamma(u_j^{\ot p^{i_{j,k}}})}\big)
\\
&= \bigsot_{j=1}^t
u_j^{\ot (\sum_{k=1}^{a_j} p^{i_{j,k}} \gamma(u_j^{\ot p^{i_{j,k}}}))}
\\
&=\bigsot_{j=1}^t
u_j^{\ot (\sum_{\ell=0}^{\infty}  p^{\ell}\gamma(u_j^{\ot p^{\ell}}))}.
\end{aligned}
\mlabel{eq:uwt}
\end{equation}
Then by Eq.~(\mref{eq:uws}),
{\allowdisplaybreaks
\begin{eqnarray}
w_\gamma
&=& \bigshprl_{j=1}^{t\ } ( N_{\gamma,u_j} u_j^{\otimes(
\sum_{\ell=0}^{\infty} p^\ell\gamma(u_j^{\otimes p^\ell}) )}+ \text{ lower
L-order terms} ) \qquad (\text{by Lemma~\ref{lem:lead}. (\mref{it:lead3})}) \notag\\
& =&N_\gamma \bigsot_{j=1}^t u_j^{\otimes( \sum_{\ell=0}^{\infty}
p^\ell \gamma(u_j^{\otimes p^\ell}) )} +\text{ lower L-order terms}  \qquad (\text{by
Lemma~\ref{lem:lead}.
(\mref{it:lead1})})\mlabel{eq:standbase}\\
&=& N_\gamma w_1^{\otimes \gamma(w_1)}\otimes \cdots
\otimes w_r^{\otimes \gamma(w_r)}+\text{ lower L-order terms}
\qquad (\text{by Eq.~(\mref{eq:uwt})}) .
\notag
\end{eqnarray}}
Here $N_{\gamma,u_j}$ is a $p$-adic unit that only
depends on $u_j$ and $\gamma$, and $N_\gamma=\prod_{j=1}^t N_{\gamma,u_j}$. Since all the leading terms are
distinct and the leading coefficients are $p$-adic units, the
displayed elements in $U$ are all distinct.

Now suppose the set $U$ is linearly dependent. Then there is a
linear combination
$$ \sum_{u\in U} a_u u =0$$
such that not all $a_u$ are zero. Among all the $u$'s with nonzero
coefficients, let $u_0$ be the one such that the leading word $w$ of $u_0$ in Eq.~(\mref{eq:standbase}) is the largest. Then $a_{u_0}$ is in fact the coefficient of $w$ when $\sum_{u\in U} a_u u=0$ is expanded by Eq.~(\mref{eq:standbase}). Therefore $u_0=0$, a contradiction.
\end{proof}

\subsection{Mixable shuffle algebras with coefficients in $\FF_p$}
\mlabel{ss:shp}
Let $p$ be a prime and let $\bfk=\FF_p$ in this section. We study the structure of $\sh_{\bfk,\lambda}(S)$ for a semigroup $S$. When $\lambda=0$, this structure is easy to give (Theorem~\mref{thm:psh}). It is more subtle when $\lambda\neq 0$ and we have to distinguish several types of abelian semigroups, such as free semigroups, elementary $p$-groups and $p$-idempotent semigroups. To avoid case by case consideration and repeated arguments, we provide an axiomatic framework in Section~\mref{sss:sg} before stating and proving our main theorem in Section~\mref{sss:wtL}.

\subsubsection{Mixable shuffle algebras of weight $0$}
\mlabel{ss:wt0}
We consider mixable shuffle algebras $\sh_{\FF_p,\lambda}(S)$ of weight $0$, that is, shuffle product algebras. It is defined as long as $S$ is a set.
\begin{defn}
Let $A$ be a $\bfk$-algebra. Let $Y$ be a subset of $A$. Define $$\sym{Y}:=\{\sym{y}\ |\ y\in Y\}$$
to be the set of symbols that is in bijection with $Y$. Define
$$\phi: \bfk[\sym{Y}]\to A,\quad \sym{y}\mapsto y, \quad y\in Y,$$
to be the algebra homomorphism that ``evaluates" $\sym{y}$ to $y$.
\mlabel{de:sym}
\end{defn}

\begin{theorem} Let $S$ be a finite ordered set.
Let $\TL=\rmt(\Lyn(S))$ be as defined in Eq.~(\mref{eq:check}). Let $\sym{\TL}=\{\sym{w}\ |\ w\in \TL\}$ be as defined in Definition~\mref{de:sym}. Then
\begin{equation}
\sh_{\FF_p,0}(S) \cong \FF_p [\sym{\TL}]/\langle\sym{w}^p\ |\ \sym{w}\in \TL\rangle
=\bigot_{\sym{w}\in \TL} \Big(\FF_p[\sym{w}]/\langle \sym{w}^p\rangle\Big).
\mlabel{eq:psh}
\end{equation}
Here $\ideal{Y}$ denotes the ideal generated by $Y$.
\mlabel{thm:psh}
\end{theorem}
\delete{
\begin{remark}
In order not to introduce too many symbols, we have use the same notation for the subset $\TL$ of $\sh_{\FF_p,0}(S)$ and the set of symbols in bijection with $\TL$. So to be complete
\mlabel{rk:sym}
\end{remark}
}
\begin{proof}
By Proposition~\mref{pp:onto}.(\mref{it:span}),
we have a surjective $\FF_p$-algebra homomorphism
$$ \phi: \FF_p[\sym{\TL}] \to \sh_{\FF_p,0}(S), \sym{w}\mapsto w, w\in \TL.$$
As remarked at the beginning of Section~\mref{sec:chp}, $u^p=p! u^{\ot p}=0$ for any word $u$ in $\sh_{\FF_p,0}(S)$.
Thus $\langle \sym{w}^p\ |\ \sym{w}\in \sym{\TL}\rangle$ is in the kernel of
$\phi$. Note that the set
$$\{1\}\cup \{\sym{w}_1^{n_1} \cdots \sym{w}_r^{n_r} |\ \sym{w}_j\in \sym{\TL}, w_1>\cdots>w_r, 1\leq n_j\leq p-1, 1\leq j\leq r, r\geq 1\}
$$
is a $\FF_p$-basis of $\FF_p [\sym{\TL}]/\langle\sym{w}^p\ |\ \sym{w}\in \TL\rangle$ which is mapped onto the subset
$$U=\{1 \}\cup \{w_1^{\shpr n_1}\shpr \cdots \shpr w_r^{\shpr n_r} |\ w_j\in \TL, w_1>\cdots>w_r, 1\leq n_j\leq p-1, 1\leq j\leq r, r\geq 1\}
$$
of $\sh_{\FF_p,0}(S)$ defined in Eq.~(\mref{eq:u}).
Thus to show that $\phi$ is injective and
hence finish the proof of the theorem, we only need to show that $U$ is linearly independent. This is just Proposition~\mref{pp:onto}.(\mref{it:ind}).
\end{proof}

\subsubsection{Two classes of semigroups and their Lyndon words}
\mlabel{sss:sg}

For an abelian semigroup $\sg$, define
\begin{equation}
\sg_1=\{g\in\sg\ |\  g^p=g\}, \quad S_2=\{g\in \sg\ |\ g^p\neq g\}.
\mlabel{eq:s1}
\end{equation}
Then $\sg=\sg_1\coprod \sg_2$.
We will study $\sh_{\FF_p,\lambda}(S)$ for $S$ in the following two classes of abelian semigroups.

\begin{defn}
{\rm
\begin{enumerate}
\item
Let $\pG$ denote the class of well-ordered abelian semigroups $(\sg,<)$ such that, for any $a,b\in \sg$,
\begin{eqnarray}
&&a>b \Rightarrow a^p>b^p, \text{\rm and}
\mlabel{eq:ple}
\\
&&a^p\geq a.
\mlabel{eq:pleq}
\end{eqnarray}
\item Let $\jG$ denote the class of well-ordered abelian semigroups $(\sg,<)$ such that every element $g\in\sg$ satisfies $g^{p^2}=g^p$
and $g_1<g_2$ for $g_1\in \sg_1$ and $g_2\in \sg_2$.
\end{enumerate}
}
\mlabel{de:class}
\end{defn}

We give some examples to illustrate the wide range of semigroups covered by these two classes. We start with some examples and properties of $\pG$.
\begin{prop}
\begin{enumerate}
\item
$\pG$ contains the class $\iG$ of pairs $(\sg,<)$ consisting of a finite abelian semigroup $\sg$ that is {\bf $p$-idempotent}, that is, $g^{p}=g$ for any element $g$ in the semigroup, and any well order $<$ on $\sg$.
\mlabel{it:ig}
\item
Let $\fG$ be the class of free abelian semigroups $\fg=\fg(X)$
generated by ordered finite sets $X$. For
$(x_1^{n_1},\cdots,x_{|X|}^{n_{|X|}})\in \fg, x_i\in X, n_i\geq 1,
1\leq i\leq |X|,$ define
$\deg(x_1^{n_1},\cdots,x_{|X|}^{n_{|X|}})=\sum_{i=1}^{|X|} n_i.$ For
$y_1, y_2 \in \fg$, define $y_1> y_2$ if $\deg(y_1)>\deg(y_2)$, or
if $\deg(y_1)=\deg(y_2)$ and $y_1$ is larger than $y_2$ according to
the lexicographic order on $\fg$ induced by the order on $X$. Then
$\fG$ is a subclass of $\pG$. \mlabel{it:fg}
\item
The class $\pG$ is closed under the semigroup unitarization that adds an identity $\iota_\pg$ to an ordered semigroup $\pg\in \pG$. The order on $\pg$ is extended to $\pg\cup \{\iota_\pg\}$ by defining $\iota_\pg$ to be the smallest element. In particular,
$\pG$ contains free abelian monoids $M\com(X)$ generated by ordered  finite sets $X$.
\mlabel{it:mg}
\item
The class $\pG$ is closed under taking finite direct products and sub-objects, with the (lexicographic) product order and restricted order, respectively.
\mlabel{it:subp}
\item
The class $\pG$ is closed under taking semigroup direct coproducts with the coproduct order (see the proof for the construction).
\mlabel{it:cg}
\end{enumerate}
\mlabel{pp:pexam}
\end{prop}

\begin{proof}
(\mref{it:ig}). Both of the two conditions on $\pG$ follow from
the $p$-idempotent condition $g^p=g$.
\smallskip

\noindent
(\mref{it:fg}). Here checking of the two conditions boils down to the facts that, for positive integers $m, n$,
$m>n$ if and only if $pm>pn$, and that $pm>m$.
\smallskip

\noindent
(\mref{it:mg}) Let $\pg\in \pG$ and consider the monoid $\pg\cup \{\iota_\pg\}$. Since elements in $\pg$ already satisfy the two conditions for $\pG$ and there is no $a\in \pg$ with $\iota_\pg>a$, we only need to check that $a>\iota_\pg$ implies $a^p>\iota_\pg^p$ and that $\iota_\pg^p\geq \iota_\pg$, both of which are clear.
\smallskip

\noindent
(\mref{it:subp}) holds since the two conditions on $\pG$ are preserved by taking finite direct products and subsets.
\smallskip

\noindent
(\mref{it:cg}). Let $\sg, {\sg'}\in \pG$.
The coproduct $\cg=\cg(\sg,{\sg'})$ of $\sg$ and ${\sg'}$ is defined by the usual universal property. Explicitly, $C$ is the disjoint union
$$ C=(\sg\times {\sg'})\coprod \sg \coprod {\sg'}.$$
Extending the semigroup $\sg$ (resp. ${\sg'}$) to the monoid $\sg\cup\{\iota_\sg\}$ (resp. ${\sg'}\cup \{\iota_{\sg'}\}$) by adding an identity $\iota_\sg$ (resp. $\iota_{\sg'}$).
Thus we can rewrite $\cg$ as the sub-semigroup
$$\cg=\{(y,g)\in (\sg \cup\{\iota_\sg\})\times ({\sg'}\cup \{\iota_{\sg'}\}) \ |\ (y,g)\neq (\iota_\sg, \iota_{\sg'}) \}$$
of the monoid product $(\sg \cup\{\iota_\sg\})\times ({\sg'}\cup \{\iota_{\sg'}\})$. By Item~(\mref{it:mg}), $\sg \cup\{\iota_\sg\}$ and ${\sg'}\cup \{\iota_{\sg'}\}$ are in $\pG$. Hence by Item~(\mref{it:subp}), $\pG$ contains $(\sg\cup\{\iota_\sg\}) \times ({\sg'} \cup \{\iota_{\sg'}\})$ with the product order, and then contains $C\subseteq (\sg\cup\{\iota_\sg\}) \times ({\sg'} \cup \{\iota_{\sg'}\})$ with the restricted order.
\end{proof}
We next provide some examples and properties of $\jG$.
\begin{prop}
\begin{enumerate}
\item
Let $\iG$ be the class in Proposition~\mref{pp:pexam}.(\mref{it:ig}).  Then $\iG\subseteq \jG$.
\mlabel{it:ig2}
\item
$\jG$ contains the class $\eG$ of pairs $(\sg,<)$ consisting of a
finite abelian group $\sg$ that is an {\bf elementary $p$-group}, that
is, $g^p=e$ for any element in the group. Here $e$ is the
identity and $<$ is any choice of well order on $\sg$ such that $e$ is the smallest element.
\mlabel{it:eg}
\item
The class $\jG$ is closed under taking finite direct products and sub-objects, with the product order and restricted order, respectively.
\mlabel{it:ieg}
\end{enumerate}
\end{prop}
We will use the notations $\pG, \jG, \iG, \fG, \cG, \eG$ with
the above meanings in the rest of this paper.

\begin{proof}
The verifications of Items (\mref{it:ig2}) and (\mref{it:eg}) are clear. Item (\mref{it:ieg}) follows since the defining properties of $\jG$ are preserved under taking finite direct products and subsets.
\end{proof}

Let a semigroup $\sg$ be in $\pG$ or $\jG$. For a \word
$w=u_1\ot \cdots \ot u_r\in \sg^{\ot r}\subseteq M\tcon(S)$, denote
\begin{equation}
w^{\comp{p}}=u_1^p\ot \cdots \ot u_r^p. \mlabel{eq:comp}
\end{equation}

\begin{lemma} Let $\sg\in \pG$.
\begin{enumerate}
\item
$a>b \Leftrightarrow a^p>b^p$.
\mlabel{it:iff1}
\item
$a=b \Leftrightarrow a^p=b^p$.
\mlabel{it:iff2}
\item
A \word $w\in M\tcon(\sg)$ is a Lyndon word if and only if $w^{\comp{p}}$ is a Lyndon word.
\mlabel{it:iff3}
\end{enumerate}
\mlabel{lem:iff}
\end{lemma}
\begin{proof}
(\mref{it:iff1}). Suppose $a^p>b^p$ but $a\leq b$, then either $a<b$ which implies that $a^p<b^p$, or $a=b$ which implies that $a^p=b^p$. Both are contradictions. So $a>b$. The same argument applies to prove (\mref{it:iff2}).
\medskip

\noindent
(\mref{it:iff3}). By Items (\mref{it:iff1}) and (\mref{it:iff2}),
the map
$$F: \sg\to \sg':=\{g^p\ |\ g\in \sg\},\quad g\mapsto g^p, g\in \sg,$$ is an isomorphism of the two ordered sets with the order on $\sg'$ being restricted from $\sg$. Since Lyndon words are determined solely by the orders, an order-preserving set map sends a Lyndon word to a Lyndon word. Then Item~(\mref{it:iff3}) follows.
\end{proof}

For $S\in \pG$ or $\jG$,  $\sg_1$ is a sub-semigroup of $\sg$ and remains in the same class as $\sg$. Define the subset of
{\bf $p$-divisible elements} of $\sg$:
\begin{equation}
\sg_\mdiv: = \bigcap_{r\geq 1} \{u^{p^r}\ |\ u\in \sg\}.
\mlabel{eq:div}
\end{equation}
\begin{lemma}
Let $S$ be in $\pG$.
\begin{enumerate}
\item
$ \sg_\mdiv = S_1.$
\mlabel{it:div}
\item
For $i=1,2$, $g\in \sg_i$ if and only if $g^p\in \sg_i$.
\mlabel{it:disj}
\end{enumerate}
\mlabel{lem:div}
\end{lemma}
\begin{proof}
(\mref{it:div}).
Since clearly $S_\mdiv \supseteq S_1$, it remains to show that
$ S_\mdiv \backslash S_1$
is empty. Suppose not, then since $\sg$ is a well-ordered set,
$S_\mdiv \backslash S_1$ has a minimal element, denoted by $w_0$.
Then $w_0\neq w_0^p$ but $w_0=u^p$ for some $u\in S$.
Since $w_0$ is in $S_\mdiv$, there is a $u_r$ for each $r\geq 1$ such that $w_0=u_r^{p^r}.$ Then we have $u_1^p=w_0=(u_r^{p^{r-1}})^p$ for $r\geq 2$. By Lemma~\mref{lem:iff}.(\mref{it:iff2}), we get
$u_1=u_r^{p^{r-1}}, r\geq 2.$ Thus $u_1$ is in $S_\mdiv$. Suppose $u_1$ is in $S_1$. Then $u_1=u_1^p=w_0$. Then $w_0^p=u_1^p=u_1=w_0$, yielding a contradiction. Therefore, $u_1\in S_\mdiv\backslash S_1$.
By the minimality of $w_0$,
we must have $w_0\leq u_1$. By Eq.~(\mref{eq:pleq}), $w_0=u_1^p\geq u_1$. Thus
$w_0=u_1$, that is, $w_0=u_1^p=w_0^p$, again a contradiction. \medskip

\noindent
(\mref{it:disj}).
By Lemma~\mref{lem:iff}.(\mref{it:iff2}),
$$ g\in \sg_1 \Leftrightarrow g^p=g \Leftrightarrow g^{p^2}=g^p
\Leftrightarrow g^p\in \sg_1.$$
Then the claim for $\sg_2$ follows since $S_1$ and $S_2$ are disjoint.
\end{proof}

We define the following operators on subsets $W\subseteq M\tcon(S).$
\begin{equation}
\begin{aligned}
W_1&=\{ w\in W\ |\ w^{\comp{p}}=w\}, \\
W_2&=\{ w\in W\ |\ w^{\comp{p}}\neq w\},\\
\rme(W)&=\{ w\in W\ |\ \text{ either } w= w^{\comp{p}} \text{ or } w \neq u^{\comp{p}} \text{ for any }u\in M\tcon(S)\}.
\end{aligned}
\mlabel{eq:12}
\end{equation} Clearly $W=W_1\coprod W_2$.
Recall from Eq.~(\mref{eq:check}) that we have also defined the operator
$$\rmt(W)=\{ w^{\ot{p^i}}\ |\ i\in \ZZ_{\geq 0}, w\in W\}.$$

The following lemma shows that the four operators $W\mapsto W_1, W\mapsto W_2, W\mapsto \rme(W)$ and $W\mapsto \rmt(W)$ all commute with one another.

\begin{lemma} Let $W$ be any subset of $M\tcon(S)$.
{\allowdisplaybreaks
\begin{eqnarray}
&\rme ( \rmt (W)) = \rmt ( \rme (W)), \quad
\rmt(W_i)=\rmt(W)_i, \quad i=1,2. \mlabel{eq:comm}\\
& \rme(W_1)=W_1=\rme(W)_1, \quad \rme(W_2)=\rme(W)_2.
 \mlabel{eq:rmet}
\end{eqnarray}
}
\mlabel{lem:op1}
\end{lemma}
\begin{proof}
Eq.~(\mref{eq:comm}) follows easily from the definitions.

For Eq.(\mref{eq:rmet}), $\rme(W_1)=W_1$ follows from the definitions.
Then $$W_1=(W_1)_1=\rme(W_1)_1 \subseteq \rme(W)_1 \subseteq W_1.$$
Thus
$\rme(W_2)=\rme(W\backslash W_1)=\rme(W)\backslash \rme(W_1)=
\rme(W)\backslash \rme(W)_1=\rme(W)_2.$
\end{proof}

For notational convenience, we will skip the parentheses in the operators and denote
\begin{eqnarray*}
&\rme W=\rme(W), \quad  \rmt W=\rmt(W), \quad  \rmt \rme
W=\rmt(\rme(W)), \\
&\rme W_i=\rme (W_i),  \quad \rmt W_i=\rmt (W_i), \quad  \text{TE}W_i=\rmt (\rme (W_i)),  \quad i=1,2.
\end{eqnarray*}
In particular, for $\rml=\Lyn(S)$,
\begin{eqnarray}
&\EL =\rme(\rml), \quad \TL=\rmt(\rml), \quad \TEL=\rmt(\rme(\rml)),\notag \\
& \EL_i=\rme (\rml_i), \quad \TL_i=\rmt(\rml_i), \mlabel{eq:esl12}\\
&\TEL_i  =\rmt (\rme (\rml_i))=\{w=u^{\otimes p^r}\ |\  u\in\EL_i,
r\in\mathbb{Z}_{\geq 0}\}, \quad  i=1,2.  \notag
\end{eqnarray}
By Lemma~\mref{lem:op1}, there is no ambiguity in these notations, since, for example,
$$
\TEL_1=\rmt(\rme(\rml_1))=\rmt(\rme(\rml)_1)=\rmt(\rme(\rml))_1
=\rme(\rmt(\rml_1)).
$$
When $S$ is the free abelian semigroup with one generator, our $\TL$ and $\TEL$ agree with the sets $SL$ and $ESL$ defined in~\mcite{Ha}.

\begin{lemma} Let $\sg$ be in $\pG$. Let $\rml=\Lyn(S)$ be the set of Lyndon words. Then we have
{\allowdisplaybreaks
\begin{eqnarray}
\TL_1&=&\TEL_1, \mlabel{eq:esl1}\\
\rml_1 &= &
    \Lyn(\sg_1).
\mlabel{eq:l1i}
\end{eqnarray}
}
\mlabel{lem:op}
\end{lemma}

\begin{proof}
By Eq.~(\mref{eq:rmet}) we have $\rml_1= \EL_1$. So applying the operator $\rmt$, we have
$\TL_1=\TEL_1$.

For Eq.~(\mref{eq:l1i}), let $w=w_1\ot \cdots \ot w_r\in S^{\ot r}$ be a Lyndon word. Then
\begin{eqnarray*}
&(w\in \rml_1) \Leftrightarrow (w^{\comp{p}}=w)
\Leftrightarrow (w_1^p \ot \cdots \ot w_r^p=w_1\ot \cdots \ot w_r) \\
&\Leftrightarrow (w_i^p =w_i, 1\leq i\leq r)
\Leftrightarrow (w_i \in S_1, 1\leq i\leq r)
\Leftrightarrow (w\in \Lyn(S_1)).
\end{eqnarray*}
\end{proof}

\begin{lemma} Let $\sg\in \pG$.
Then $\TL_2=\{ u^{\comp{p^i}}\ |\ u\in \TEL_2, i\geq 0\}$. Further, all the displayed elements are distinct.
\mlabel{lem:sl}
\end{lemma}
\begin{proof}
Note that for any $u^{\comp{p^i}}$ in the set of the right hand side, $u=w^{\ot p^j}$ for some $w\in \EL_2$.
Since $(w^{\ot p^j})^{\comp{p^i}}=(w^{\comp{p^i}})^{\ot p^j}$,
and $w^{\comp{p^i}}$ is also in $\Lyn_2$ by Lemma~\mref{lem:div}.(\mref{it:disj}), we have
$(w^{\comp{p^i}})^{\ot p^j}\in \TL_2$. This proves
$\supseteq.$

Conversely, let $v^{\ot p^j}\in \TL_2$ with $v\in \Lyn_2$.
Then a tensor factor of $v$ is in $\sg_2$, and hence is not in $S_\mdiv$ by Lemma~\mref{lem:div}.(\mref{it:div}). This means $v=w^{\comp{p^i}}$ for some $w\in \EL_2$.
This shows that $v^{\ot p^j}=(w^{\comp{p^i}})^{\ot p^j}=(w^{\ot p^j})^{\comp{p^i}}$ is in $\{u^{\comp{p^i}}\ |\ u\in \TEL_2, i\geq 0\}.$

Suppose there are $u,v\in \TEL_2$ and $i,j\geq 0$ such that
$u^{\comp{p^i}}=v^{\comp{p^j}}$. Without loss of generality, we can take $i\geq j$. Then
$(u^{\comp{p^{i-j}}})^{\comp{p^j}}=v^{\comp{p^j}}$. By Lemma~\mref{lem:iff}.(\mref{it:iff2}), $u^{\comp{p^{i-j}}}=v$.
Since $v\in \TEL_2$, we have $v\neq v^{\comp{p}}$. Since  $v\in \TEL_2=\rme (\TL_2)$ by Eq.~(\mref{eq:comm}), from the definition of the operator $\rme$ in Eq.~(\mref{eq:12}), we have $v\neq w^{\comp{p}}$ for any word $w$.
So from $u^{\comp{p^{i-j}}}=v$ we obtain $i-j=0$ and then $u=v$.
\end{proof}

For $\sg\in \jG$, define
\begin{equation}
\begin{aligned}
\EETL_2:=\EETL_2(\sg)&=\{w-w^{\comp{p}}|\ w \in \TL_2\},
\\
\EETL:=\EETL(\sg)&= \TL_1(\sg)\cup \EETL_2(\sg).
\end{aligned}
\mlabel{eq:esl}
\end{equation}

\subsubsection{Mixable shuffle algebras of nonzero weight}
\mlabel{sss:wtL}
We now consider a mixable shuffle algebra $\sh_{\FF_p,\lambda}(S)$ on a semigroup $S$ when $\lambda \neq 0$.
\begin{lemma}
Let $S\in \jG$ and let $\lambda \in \FF_p$ be non-zero. For any \word $w\in \sh_{\FF_p,\lambda}(S)$,
\begin{equation}
(w-w^{\comp{p}})^{\shprl p}=0.
\mlabel{eq:shp}
\end{equation}
\mlabel{lem:shp}
\end{lemma}
\begin{proof}
Let $w$ be in $\sh_{\FF_p,\lambda}(S)$. We have
{\allowdisplaybreaks
\begin{eqnarray*}
 (w-w^{\comp{p}})^{\shprl p}&=&w^{\shprl p} -(w^{\comp{p}})^{\shprl p}\\
 &=& \lambda^{(\leng(w)-1)(p-1)}w^{\comp{p}} -\lambda^{(\leng(w^{\comp{p}})-1)(p-1)}(w^{\comp{p}})^{\comp{p}} \qquad \text{(by Eq.~(\mref{eq:41}))}\\
&=&
\lambda^{(\leng(w)-1)(p-1)}w^{\comp{p}} -\lambda^{(\leng(w)-1)(p-1)}w^{\comp{p^2}} \qquad (\leng(w^{\comp{p}})=\leng(w))\\
&=&
\lambda^{(\leng(w)-1)(p-1)}w^{\comp{p}} -\lambda^{(\leng(w)-1)(p-1)}w^{\comp{p}}. \qquad (\text{defining property of } \jG)
\end{eqnarray*}
}
Hence we have the lemma.
\end{proof}

With notations introduced in Eq.~(\mref{eq:esl12})and Eq.~(\mref{eq:esl}), we can state our main theorem on mixable shuffle algebras with weight $\lambda\neq 0$ and with coefficients in $\FF_p$.

\begin{theorem}
Let $0\neq \lambda \in \FF_p$.
We will use the notation from Definition~\mref{de:sym}.
\begin{enumerate}
\item For a semigroup $\sg$ in $\pG$, we have
\begin{eqnarray}
\sh_{\FF_p,\lambda}(\sg)&\cong& \FF_p[\TELs]/\langle \sym{w} ^p-\lambdaw \sym{w} \ |\ \sym{w} \in
\TELs_1\rangle \notag \\
&\cong& \FF_p[\TELs_1]/\langle \sym{w} ^p-\lambdaw \sym{w} \ |\ \sym{w} \in \TELs_1\rangle \ot
\FF_p[\TELs_2].
\mlabel{eq:pmsh}
\end{eqnarray}
In particular, for $S\in \fG$,
$$
   \sh_{\FF_p,\lambda}(S) \cong \FF_p[\TELs].
    $$
\mlabel{it:pmsh1}
\item
For $\sg$ in $\jG$, we have
\begin{equation}
\sh_{\FF_p,\lambda}(\sg)\cong \left(\FF_p[\TLs_1]/\langle \sym{w} ^p- \lambdaw \sym{w} \ |\ \sym{w} \in
\TLs_1\rangle\right) \otimes \left(\FF_p[\EETLs_2]/\langle \sym{w} ^p\ |\ \sym{w} \in
\EETLs_2\rangle\right). \mlabel{eq:pmsh2}
\end{equation}
\mlabel{it:pmsh2}
\end{enumerate}
\mlabel{thm:pmsh}
\end{theorem}

\begin{coro} Let $X$ be a finite ordered set. Let $S=M\com(X)$ be the free abelian monoid generated by $X$.
Then
\begin{equation}
\sh_{\FF_p,\lambda}(\FF_p[X])\cong \FF_p[\TELs_2]\ot \left (\FF_p[\TELs_1]/\langle
\sym{w} ^p-\sym{w} \ |\ \sym{w} \in \TELs_1\rangle\right ). \mlabel{eq:pmshx}
\end{equation}
\mlabel{co:pmshx}
\end{coro}
We note that in this case,
\begin{equation}
\TEL_1={\{1^{\ot p^i}\ |\ i\geq 0\}}. \mlabel{eq:l1}
\end{equation}

\begin{proof}
By Proposition~\mref{pp:pexam}.(\mref{it:mg}), $M\com(X)$ is in $\pG$. Since $\sh_{\FF_p,\lambda} (M\com(X))=\sh_{\FF_p,\lambda} (\FF_p[X]),$
the corollary follows from Theorem~\mref{thm:pmsh}.(\mref{it:pmsh1}). \end{proof}

\noindent
{\em Proof of Theorem~\mref{thm:pmsh}. }
(\mref{it:pmsh1}). We first show the surjectivity of the natural
$\FF_p$-algebra homomorphism
$$ \phi: \FF_p[\TELs] \to \sh_{\FF_p,\lambda} (\sg)$$
in Definition~\mref{de:sym} sending $\sym{w}\in \TELs$ to $w\in \TEL$.

Let
$\sh_{\FF_p,\lambda} (\sg)'$ be the image of $\phi$. By Proposition~\mref{pp:onto}, we
only need to show $\TL\subseteq \sh_{\FF_p,\lambda} (\sg)'$. Let $w\in \TL$.
Then either
$w\in \TL_1$ or $w\in \TL_2$. If $w\in \TL_1$, then by Eq.~(\mref{eq:esl1}), $w\in \TEL_1\subseteq \TEL$ and hence is in $\sh_{\FF_p,\lambda} (\sg)'$. If $w\in \TL_2$, then $w=u^{\comp{p^i}}$ for some $u\in \TEL_2$ by Lemma~\mref{lem:sl}.
By Eq.~(\mref{eq:41}),
$$ u^{\shprl p^i}= u^{\comp{p^i}}=
 w.$$
So $w$ is in
$\sh_{\FF_p,\lambda} (\sg)'$ since $u\in \TEL_2\subseteq \sh_{\FF_p,\lambda} (\sg)'$. Thus we have shown the surjectivity of $\phi$.

To prove the injectivity, first note that, by Eq.~(\mref{eq:41}),
$w^{\shprl p}=\lambdaw w$ for $w\in \TEL_1$. So the ideal $\langle
\sym{w}^{ p}-\lambdaw \sym{w}\ |\ \sym{w}\in \TELs_1\rangle$ of $\FF_p[\TELs]$ is
in $\ker(\phi)$.
Let
$$\Sigma=\{\sigma=(\sigma_1,\sigma_2)\: |\: \sigma_1: \TEL_1\rightarrow \{0,\cdots,p-1\},
\sigma_2: \TEL_2\rightarrow \mathbb{Z}_{\geq 0},\text{both with
finite supports}\}.$$
Then
$$ \sym{V}:=\Big \{\sym{z}_{\sigma}:= (\prod_{u\in\TEL_1}\sym{u}^{\sigma_1(u)})(\prod_{v\in\TEL_2} \sym{v}^{
\sigma_2(v)})
 \ |\ \sigma=(\sigma_1,\sigma_2)\in\Sigma \Big \}$$
is a $\FF_p$-basis of $\FF_p[\TELs]/\langle \sym{w}^p-\lambdaw \sym{w}\ |\ \sym{w}\in
\TELs_1\rangle$. Further,
$$
V:= \Big \{z_\sigma: =(\bigshprl_{u\in\TEL_1}u^{\shprl\sigma_1(u)})\shprl(\bigshprl_{v\in\TEL_2} v^{\shprl \sigma_2(v)})\ |\ \sigma=(\sigma_1,\sigma_2)\in\Sigma \Big \}$$
is the image of $\sym{V}$ under $\phi$.
Thus to prove the injectivity of $\phi$ we only need to show that ${V}$ is linearly independent.
For this we relate ${V}$ to the linearly independent subset $U$ defined in Eq.~(\mref{eq:u}).

Let $$\Gamma=\{\gamma: \TL\rightarrow \{0,\cdots, p-1\}\ | \ \gamma \text{ has finite support }\}.$$
Then we have
$$U=\{w_{\gamma}:=\bigshprl_{w\in
\TL}w^{\shprl\gamma(w)}\ |\  \gamma\in\Gamma\}.$$
We will construct a bijection between
$\Sigma$ and $\Gamma$.
First note that $\TL=\TL_1\coprod \TL_2=\TEL_1\coprod \TL_2$ by Eq.~(\mref{eq:esl1}) and
$$\TL_2=\{v^{\comp{p^i}}\ |\ v\in \TEL_2, i\geq 0\}$$
with all displayed elements distinct by Lemma~\mref{lem:sl}. Thus we can define
$$\eta:
\Sigma\rightarrow \Gamma, \quad \sigma\mapsto\gamma_\sigma, \sigma=(\sigma_1,\sigma_2)\in \Sigma$$
by first taking $\gamma_\sigma|_{\TEL_1}=\sigma_1$. Next for any $w
=v^{\comp{p^i}} \in \TL_2$ with $v\in \TEL_2$,  if
$\sigma_2(v)=\sum\limits_{j=0}^{\infty}a_jp^j$ with $a_j\in\{ 0,
\cdots ,p-1\}$, we define $\gamma_\sigma(w)=a_i$.
In the other direction, we define
$$\zeta: \Gamma\rightarrow\Sigma,
\quad \gamma\mapsto \sigma_\gamma=(\sigma_1,\sigma_2)$$
as follows. If $u\in \TEL_1$, then define $\sigma_1(u)=\gamma(u)$. If $v\in \TEL_2$ then $v^{\comp{p^i}}\in\TL_2$ for
all $i\geq 0$ and we define
$$\sigma_2(v)=\sum_{i=0}^{\infty}\gamma(v^{\comp{p^i}})p^i.$$
From the constructions we see that $\eta$ and $\zeta$ are inverse of
each other.

\begin{lemma} We have $V=U$.
More precisely, for any $\sigma\in \Sigma$, we have
${z}_{\sigma}=w_{\eta(\sigma)}.$
\mlabel{lem:uandv} \end{lemma}
\begin{proof} For any $v\in\TEL_2$, by Eq.~(\mref{eq:41}),
we have $$ v^{\shprl p^j}=\lambda^{j(\ell(v)-1)(p-1)}
v^{\comp{p^j}}=v^{\comp{p^j}}. $$ If $\sigma_2(v)=\sum\limits_{j=0}^{\infty} a_{v,j}
p^j$ with $a_{v,j}\in\{0,\cdots,p-1\}$, then
$$ v^{\shprl\sigma_2(v)}  = \bigshprl_{j\geq 0} (v^{\shprl p^j})^{\shprl a_{v,j}}
 = \bigshprl_{j\geq 0} (v^{\comp{p^j}})^{\shprl a_{v,j}}
$$ and so
{\allowdisplaybreaks
$$\begin{aligned}
{z}_\sigma & =
(\bigshprl_{u\in\TEL_1}u^{\shprl\sigma_1(u)})\shprl(\bigshprl_{v\in\TEL_2}v^{\shprl
\sigma_2(v)}) \\
         & =
         (\bigshprl_{u\in\TEL_1}u^{\shprl\sigma_1(u)}) \shprl(\bigshprl_{v\in\TEL_2}(v^{\comp{
p^j}})^{\shprl a_{v,j}}) \\
         &=  w_{\eta(\sigma)}.
\end{aligned}$$
}
\end{proof}

By Lemma \mref{lem:uandv} and Proposition~\mref{pp:onto}.(\mref{it:ind}), $V$ is linearly independent, as desired.
\medskip

(\mref{it:pmsh2}). Now we consider $S\in \jG$. Define
$$ \phi: \FF_p[\EETLs] \to \sh_{\FF_p,\lambda} (S)$$
to be the natural algebra homomorphism in Definition~\mref{de:sym} with $Y=\EETL$. Again let $ \sh_{\FF_p,\lambda} (S)'$
be the image.

We first prove that $\phi$ is onto. Applying
Proposition~\mref{pp:onto}.(\mref{it:span}) to the semigroup
$\sg_1$ and noting that $\TL_1=\TL(S_1)$ by applying $\rmt$ to Eq.~(\mref{eq:l1i}), we have $\sh_{\FF_p,\lambda} (\sg_1)=\phi(\FF_p[\TLs_1])$
and hence is in $\sh_{\FF_p,\lambda} (S)'$.
Now for any $w\in \TL$, either $w\in\TL_1$ or
$w=\tilde{w}+w^{\comp{p}}$ where $\tilde{w}=w-w^{\comp{p}}\in
\EETL_2\subseteq \sh_{\FF_p,\lambda} (\sg)'$ and $w^{\comp{p}}\in \sh_{\FF_p,\lambda} (S_1)= \phi(\FF_p[\TLs_1])$. Thus
$w\in \sh_{\FF_p,\lambda} (\sg)'$. Then the surjectivity follows from
Proposition~\mref{pp:onto}.(\mref{it:span}).

For $w\in \TL$, define
$$
\bar{w}:=\left\{ \begin{array}{ll} w, & w\in \TL_1,\\
    w-w^{\comp{p}}, & w\in \TL_2.
    \end{array} \right .
$$
$$ \overline{U}:=\{ 1 \}\cup \left \{\bar{w}_1^{\shprl i_1}\shprl \cdots \shprl \bar{w}_r^{\shprl i_r}\ |\
w_i\in \TL, {w}_1> \cdots >{w}_r, 1\leq i_j\leq p-1, 1\leq j\leq r, r\geq 1
    \right \}.
$$
To prove Eq.~(\mref{eq:pmsh2}), we only need to show that $\overline{U}$ is linearly independent.

Recall that the set $U$ in Eq.~(\mref{eq:u}) is just
$$ U=\{\ 1 \}\cup \left \{{w}_1^{\shprl i_1}\shprl \cdots \shprl {w}_r^{\shprl i_r}\ |\ w_i\in \TL, {w}_1>\cdots > {w}_r, 1\leq i_j\leq p-1, 1\leq j\leq r, r\geq 1    \right \}.
$$
By Eq.~(\mref{eq:standbase}), in terms of the linear representation by the standard basis of pure tensors in $\sh_{\FF_p,\lambda} (S)$,
$$ w_1^{\shprl i_1}\shprl \cdots \shprl {w}_r^{\shprl i_r}
=\mu w_1^{\ot i_1}\ot \cdots \ot {w}_r^{\ot i_r}
+ \text{lower \lord-terms }
$$
where $\mu$ is a nonzero constant.
Since $\bar{w_i}=w_i$ when $w_i\in \TL_1$ and
$\bar{w}_i=w_i-w_i^{\comp{p}}$ and $w_i^{\comp{p}}<\lengord w_i$ when
$w\in \TL_2$, we also have
$$
\bar{w}_1^{\shprl i_1}\shprl \cdots \shprl \bar{w}_r^{\shprl i_r}= \mu\,w_1^{\ot i_1}\ot \cdots \ot {w}_r^{\ot i_r}
+ \text{lower \lord-terms }
$$
for the same $\mu$ as in the last equation. It follows that $\overline{U}$ is linearly independent if and only if $U$ is linearly independent which is Proposition~\mref{pp:onto}.(\mref{it:ind}).
\proofend

\subsection{Free Rota-Baxter algebras with coefficients in $\FF_p$}
\mlabel{ss:rbafp}

We can now obtain a structure theorem on free commutative Rota-Baxter
algebras by extracting information from the structure theorem on
mixable shuffle algebras in Theorem~\mref{thm:psh}, Theorem~\mref{thm:pmsh} and Corollary~\mref{co:pmshx}.

\begin{theorem} Let $X$ be a finite ordered set. We will continue to use the $\sym{\ }$-notation in Definition~\mref{de:sym}.
\begin{enumerate}
\item
Let $\lambda=0$ and let $S=M\com(X)$ be the commutative monoid generated by $X$. Let $\TL=\rmt(\Lyn(S))$ be defined in Eq.~(\mref{eq:check}). Then
$$
\sha_{\FF_p,\lambda}(\FF_p[X]) \cong \FF_p[X] \ot \left (\FF_p[\TLs]/\langle
\sym{w}^p\ |\ \sym{w}\in \TLs\rangle\right ).
$$\mlabel{it:rbafpx0}
\item
Let $0\neq \lambda\in \FF_p$ and let $S=M\com(X)$.
Let $\TEL_2$ be as defined in Eq.~(\mref{eq:esl12}).
Then
$$\sha_{\FF_p,\lambda}(\FF_p[X]) \cong \FF_p[X\cup \TELs_2] \ot \left (\FF_p[\sym{W}]/\langle
\sym{w}^p-\lambdaw \sym{w}\, |\, \sym{w}\in \sym{W}\rangle\right ), \quad
W=\{1^{\ot p^i}\, |\, i\geq 0\}.
$$
\mlabel{it:rbafpx}
\item
Let $A=\FF_p[X]/\langle x^{p}-x\ |\ x\in X\rangle$ and $0\neq \lambda\in \FF_p$. Let $S\in \pG$ be as defined in Eq.~(\mref{eq:xi}) in the proof. Let $\TEL=\TEL(S)$ be defined in Eq.~(\mref{eq:esl12}). Then
$$
\sha_{\FF_p,\lambda}(A) \cong \FF_p[X\cup\TELs]/\langle \sym{w}^p-\lambdaw \sym{w}\ |\ \sym{w}\in X\cup \TELs\rangle.
$$
\mlabel{it:rbafpi}
\item
Let $A=\FF_p[X]/\langle x^p-1\ |\ x\in X\rangle$ and $0\neq \lambda\in \FF_p$. Let $S\in \jG$ be the abelian group $\mu_p^{|X|}$ where $\mu_p$ is the cyclic multiplicative group of order $p$. Let $\TL_1=\TL_1(S)$ be as defined in Eq.~(\mref{eq:esl12}) and let $\EETL_2=\EETL_2(S)$ be as defined in Eq.~(\mref{eq:esl}). Then
\begin{eqnarray*}
\sha_{\FF_p,\lambda}(A) &\cong &\left (\FF_p[X]/\langle x^p-1\ |\ x\in X\rangle\right) \\
&&\ot \left ( \FF_p[\TLs_1]/\langle w^p-\lambdaw w\ |\ w\in
\TLs_1\rangle \right )\otimes \left ( \FF_p[\EETLs_2]/\langle w^p\ |\
w\in \EETLs_2\rangle \right).
\end{eqnarray*}
\mlabel{it:rbafpe}
\end{enumerate}
\mlabel{thm:rbafp}
\end{theorem}

\begin{remark}
{\rm
The four cases in the theorem show quite distinct structures of free commutative Rota-Baxter algebras for different weights and generating algebras $A$. First of all, when the weight is zero, then the polynomial part of $\sha_{\FF_p,0}(X)$ is $\FF_p[X]$ itself. The second tensor factor (the shuffle algebra part) is completely nilpotent.

In the case of $\lambda\neq 0$, when $A=\FF_p[X]$, $\sha_{\FF_p,\lambda}(A)$ is basically a free (i.e., polynomial) $\FF_p$-algebra except the subalgebra
$$ \bigoplus_{k\geq 1} \FF_p 1^{\ot k}\cong
\FF_p[\sym{W}]/\langle
\sym{w}^p-\lambdaw \sym{w}\, |\, \sym{w}\in \sym{W}\rangle,\quad
W=\{1^{\ot p^i}\, |\, i\geq 0\}.$$
When $A=\FF_p[x]/\langle x^p-x\rangle$, even though the corresponding free commutative Rota-Baxter algebra does not have any polynomial part, its structure reflects its base algebra
in the sense that
$$\sha_{\FF_p,\lambda}(A)
\cong \FF_p[\{x\}\cup \TELs]/\langle \sym{w}^p-\lambdaw \sym{w}\ |\ \sym{w}\in \{x\}\cup \TELs\rangle
\cong \bigotimes_{i\in \{x\}\cup \TELs} A_i,
 A_i\cong A, \forall i\in \{x\}\cup \TELs,
$$
is just a tensor product of copies of $A$.
In this sense, when $A=\FF_p[x]/\langle x^p-1\rangle$,
the structure of $\sha_{\FF_p,\lambda}(A)$ has completely diverged from $A$ since the only part of $\sha_{\FF_p,\lambda}(A)$ that is isomorphic to $A$ is the first tensor factor contributed from
$\sha_{\FF_p,\lambda}(A)=A\ot \sh_{\FF_p,\lambda} (A)$.
Such diversities can be expected in other free commutative Rota-Baxter algebras.
}
\end{remark}

\begin{proof}
We recall the tensor decomposition of the free commutative Rota-Baxter algebra on an algebra $A$ in Eq.~(\mref{eq:freerb}):
$$ \sha_{\FF_p,\lambda}(A)= A\ot \sh_{\FF_p,\lambda} (A).
$$
Then Item~(\mref{it:rbafpx0}) follows from Theorem~\mref{thm:psh}.
Item~(\mref{it:rbafpx}) follows from Corollary~\mref{co:pmshx}.

For (\mref{it:rbafpi}), consider the cyclic group of order $p-1$,   $\mu_{p-1}=\{\xi,\xi^2,\cdots,\xi^{p-1}\}$ where $\xi^{p-1}$ is the identity. Define $G=\{e\}\cup \mu_{p-1}$ to be the monoid from the unitarization of $\mu_{p-1}$. So the multiplication on $G$ is extended from $\mu_{p-1}$ by
$$ e\cdot e =e, e\cdot \xi^i = \xi^i=\xi^i\cdot e, 1\leq i\leq p-1.$$
It is clear that the algebra homomorphism
$$f: \FF_p[x]\to \FF_p G, \quad x\mapsto \xi$$
has $\langle x^p-x\rangle$ in its kernel. It is surjective since $f(x^i)=\xi^i$, $1\leq i\leq p-1$, and $f(1)=e$. Then
$\FF_p[x]/\langle x^p-x\rangle \cong \FF_p G$ since both $\FF_p$-algebras have the same dimension.
Now $G$, and hence
\begin{equation}
S: =G^{|X|},
\mlabel{eq:xi}
\end{equation}
 are in the class $\pG$. Then Item~(\mref{it:rbafpi}) follows from Theorem~\mref{thm:pmsh}.(\mref{it:pmsh1}).

Finally Item~(\mref{it:rbafpe}) follows from Theorem~\mref{thm:pmsh}.(\mref{it:pmsh2}).
\end{proof}

\section{Structure theorems on $\ZZ_p$}
\mlabel{sec:zp}
We now lift our Theorem~\mref{thm:pmsh} for mixable shuffle algebras in Section~\mref{sec:chp} from $\FF_p$ to $\ZZ_p$ by the Nakayama Lemma and a topological consideration. We then obtain a canonical polynomial algebra in the free commutative Rota-Baxter $\ZZ_p$-algebra generated by a finite set.

\subsection{Mixable shuffle algebras with coefficients in $\ZZ_p$}
We first recall notations and properties of graded sets and their polynomial algebras. Let $Y=\coprod_{n\ge 1}Y^{(n)}$ be a graded set. We  define the degree of $y\in Y^{(n)}$ by $\deg(y)=n$. Let $F(Y)$ be the free abelian semigroup generated by $Y$. For $y=y_1\cdots y_k\in F(Y)$ with $y_j\in Y, 1\leq j\leq k$, define $\deg(y)=\deg(y_1)+\cdots + \deg(y_k)$. In this way, the polynomial algebra $\bfk[Y]$ over a commutative ring $\bfk$ becomes a graded algebra:
$\bfk[Y]=\oplus_{n\ge 0} \bfk[Y]^{(n)}$.

\begin{lemma} Let $Y=\coprod_{n\geq 0}Y^{(n)}$ be a graded set.
\begin{enumerate}
\item
For any $n\ge 1$, as a $\bfk$-module,
$$\bfk[Y]^{(n)}= (\sum_{j=1}^{n-1}\bfk[Y]^{(j)} \bfk[Y]^{(n-j)}) \oplus \bfk Y^{(n)}.$$
\mlabel{it:grad1}
\item
Let $R=\oplus_{n\geq 0} R^{(n)}$ be a graded algebra and let $T$ be a graded subset of $R$. Let $\sym{T}$ be a set that is in bijection with $T$ and is equipped with the grading from $T$. Then the homomorphism $\phi: \bfk[\sym{T}] \to R$ in Eq.~(\mref{de:sym}) is a graded algebra homomorphism.
\mlabel{it:grad2}
\end{enumerate}
\mlabel{lem:grad}
\end{lemma}
\begin{proof}
(\mref{it:grad1}).
The degree on $F(Y)$ makes $F(Y)$ into a graded semigroup and
$\bfk[Y]^{(n)}=\bfk F(Y)^{(n)}.$ Then the lemma follows from the disjoint union decomposition
$$ F(Y)^{(n)}= (\cup_{j=1}^{n-1} F(Y)^{(j)}F(Y)^{(n-j)}) \coprod Y^{(n)}$$
of $F(Y)^{(n)}$ into elements of $Y$ and elements which are products of at least two elements of $Y$.
\smallskip

\noindent
(\mref{it:grad2}) is the universal property of $\bfk[\sym{T}]$ as the free commutative algebra generated by the graded set $\sym{T}$~\cite[Proposition 3.1]{Mac}.
To be explicit, $\phi$ preserves the gradings when it is restricted to $\sym{T}$. Since the grading on any graded algebra is multiplicative,
the grading preserving map $\phi: \sym{T} \to T$ extends to a grading preserving homomorphism $\phi: \QQ[\sym{T}]\to R.$
\end{proof}

Consider $S\in \fG$, that is, $S$ is a free abelian semigroup generated by an ordered finite set. We will continue to use the total degree on $S$ defined in Proposition~\mref{pp:pexam}.(\mref{it:fg}).
For a word $w=w_1\otimes \cdots \otimes w_r \in S^{\ot r}\subseteq \sh_{\bfk,\lambda}  (\sg)$, we define the degree of $w$ by
\begin{equation}
\deg(w)=\deg(w_1)+\cdots +\deg(w_r).
\mlabel{eq:shdeg}
\end{equation}
Then $\sh_{\QQ,\lambda}(S)$ is a graded algebra by the same argument as that in~\cite[Theorem 2.1]{Ho} where the case $\lambda=1$ is considered.
Note that
\begin{equation}
\deg(w^{\comp{p}})=\deg(w^{\otimes
p})=p\:\deg(w).
\mlabel{eq:degp}
\end{equation}

Let $\Lyn^{(n)}=\Lyn(S)^{(n)}$ be the subset of Lyndon words on $S$ of degree $n$. Since
all elements in $S$ have positive degrees, $\Lyn^{(n)}$ is finite
for each $n\geq 1$. So we have a graded set $\Lyn=\coprod_{n\geq
1}\Lyn^{(n)}$ with each homogeneous component finite.
By applying Lemma~\mref{lem:grad}.(\mref{it:grad2}), Theorem~\mref{thm:msq} has the following refined form.
\begin{theorem}
Let $S$ be in $\fG$ and let $\lambda$ be in $\QQ$. Then the inclusion map $\Lyn(S) \subseteq \sh_{\QQ,\lambda}(S)$ induces an isomorphism
$f: \QQ[\Lyn(S)] \to \sh_{\QQ,\lambda}(S)$ of graded algebras. Here the grading on $\QQ[\Lyn(S)]$ is given by the graded set $\Lyn(S)$.
\mlabel{thm:msq2}
\end{theorem}

Now we consider $\sh_{\ZZ_p,\lambda} (\sg)$ defined
over $\ZZ_p$.

\begin{prop}
Let $\lambda$ be a unit in $\ZZ_p$. For $\sg$ in $\fG$ (resp. in $\jG$) from Proposition~\mref{pp:pexam} (resp. Definition~\mref{de:class}), the natural homomorphism from Definition~\mref{de:sym}
$$\phi :
\ZZ_p[\TELs] \rightarrow \sh_{\ZZ_p,\lambda} (\sg), \quad \sym{w}\mapsto w,$$
$$ {\rm (} resp. \quad \phi :
\ZZ_p[\EETLs] \rightarrow \sh_{\ZZ_p,\lambda} (\sg), \quad \sym{w} \mapsto w {\rm )}$$
 is surjective.
\mlabel{prop:surj}
\end{prop}
\begin{proof}
We first consider $\sg\in \fG$. In this case $\sg$ is the free abelian
semigroup generated by a finite set.
By Lemma~\mref{lem:grad}.(\mref{it:grad2}), $\phi$ is a homomorphism of graded algebras. Its reduction modulo $p$ gives the graded algebra homomorphism
$$\bar{\phi }: \mathbb{F}_p[\TELs]\rightarrow
\sh_{\FF_p,\bar{\lambda}}(\sg).$$
Here $\bar{\lambda}$ is $\lambda\mod p$. By Theorem \ref{thm:pmsh}, $\bar{\phi}$ is an isomorphism. Therefore the
map of $\mathbb{F}_p$-vector spaces
$$\bar{\phi }^{(n)}: \mathbb{F}_p[\TELs]^{(n)}\rightarrow \sh_{\FF_p,\bar{\lambda}}(\sg)^{(n)}$$
is isomorphic and in particular is surjective. Since for $\sg\in \fG$, the number of elements of fixed degree is finite, the number of \words from $S$ of fixed degree is finite. Thus both
$\ZZ_p[\TELs]^{(n)}$ and $\sh_{\FF_p,\lambda} (\sg)^{(n)}$ are
of finite rank over $\ZZ_p$. Then by Nakayama Lemma the map
$$\phi ^{(n)}: \ZZ_p[\TELs]^{(n)}\rightarrow \sh_{\FF_p,\lambda} (\sg)^{(n)}$$
is surjective. This implies that $\phi $ is surjective for $\sg\in\fG$.
\smallskip

We next consider the case of $\sg\in \jG$. Applying
Proposition~\mref{pp:onto}.(\mref{it:span}) to the semigroup
$\sg_1$ and noting that $\TL_1=\TL(S_1)$ by applying $\rmt$ to Eq.~(\mref{eq:l1i}), we have $\sh_{\ZZ_p,\lambda} (\sg_1)=\phi(\ZZ_p[\TLs_1])$
and hence is in $\sh_{\ZZ_p,\lambda} (S)'$.
Now for any $w\in \TL$, either $w\in\TL_1$ or
$w=\tilde{w}+w^{\comp{p}}$ where $\tilde{w}=w-w^{\comp{p}}\in
\EETL_2\subseteq \sh_{\ZZ_p,\lambda} (\sg)'$ and $w^{\comp{p}}\in \sh_{\ZZ_p,\lambda} (S_1)= \phi(\ZZ_p[\TLs_1])$. Thus
$w\in \sh_{\ZZ_p,\lambda} (\sg)'$. Then the surjectivity follows from
Proposition~\mref{pp:onto}.(\mref{it:span}).
%
\end{proof}

For $\sg\in \jG$, let $\sym{v} \in \TLs_1$ and $\sym{w}\in \EETLs_2$. Then by Theorem
\ref{thm:pmsh} we have
$$\phi (\sym{v})^{\shprl p}, \phi (\sym{w})^{\shprl p}-\lambdaw\phi (\sym{w})\in
p\sh_{\ZZ_p,\lambda} (\sg).$$ By Proposition \ref{prop:surj} there are
polynomials $Q'_v$ and $Q'_w$ in $\ZZ_p[\EETLs]$ such that
$$\phi (\sym{v})^{\shprl p} =p\phi (Q'_v)\;\;\;\;\phi (\sym{w})^{\shprl p} -\lambdaw \phi (\sym{w})=p\phi (Q'_w).$$ Thus
$$Q_v:=\sym{v}^{p} -pQ'_v, \;\;\; Q_w:=\sym{w}^{p}-\lambdaw \sym{w}-pQ'_w$$
are in $\ker \phi $.
Let $I$ be the ideal of $\ZZ_p[\EETLs]$ generated by the $Q_v$'s and
$Q_w$'s. Then $I\subseteq \ker \phi$. Let $\bar{I}$ be the  closure of $I$ in $\ZZ_p[\EETLs]$ with
respect to the $p$-adic topology, that is,
$$\bar{I}=\bigcap_{n\geq 0}(I+p^n \ZZ_p[\EETLs]).$$
Then the modula $\ZZ_p[\EETLs]/\bar{I}$ is {\bf separated
with the $p$-adic topology}, i.e.
$$\bigcap_{n\geq 0} p^n (\ZZ_p[\EETLs]/\bar{I})=0. $$

Because
$\bar{I}\subset I+p^n \ZZ_p[\EETLs], n\geq 0,$ we have
$$\phi (\bar{I})\subseteq p^n \sh_{\ZZ_p,\lambda} (\sg).$$ So
$\phi (\bar{I})\subseteq \bigcap\limits_{n\geq 0} p^n
\sh_{\ZZ_p,\lambda} (\sg).$ Since $\sh_{\ZZ_p,\lambda} (\sg)$ is a free
$\ZZ_p$-module, we have $\bigcap\limits_{n\geq 0} p^n
\sh_{\ZZ_p,\lambda} (\sg)=0.$ Hence $\phi (\bar{I})=0$.
Thus $\phi $ induces a homomorphism
$$\ZZ_p[\EETLs]/\bar{I}\rightarrow \sh_{\ZZ_p,\lambda} (\sg),$$ which is
again denoted by $\phi $.
We give a lemma before presenting our main theorem in this section.
\begin{lemma} Let $M$ be a $\ZZ_p$-module that is separated for the
$p$-adic topology and let $N$ be a torsion-free $\ZZ_p$-module. Let
$f: M\rightarrow N$ be a homomorphism of $\ZZ_p$-modules.
If the induced homomorphism
$$ \bar{f}: M \ot\FF_p \rightarrow N \ot \FF_p$$ is injective,
then $f$ is also injective.
\mlabel{lem:inj}
\end{lemma}
\begin{proof} Let $m\in \mathrm{ker}(f)$. We prove $m=0$. Since
$\bar{f}$ is an isomorphism, we have $m\in pM$. Write $m=pm_1$. Then
$f(pm_1)=pf(m_1)=0.$ Since $N$ is torsion-free, we get $f(m_1)=0$.
So we have $m_1\in pM$ and $m\in p^2 M$. An inductive argument shows
that $m\in \bigcap\limits_{n\geq 0} p^nM$. Then the condition that
$M$ is separated for the $p$-adic topology implies that $m=0$.
\end{proof}

\begin{theorem} Let $\lambda\in \ZZ_p$ be a $p$-adic unit.
\begin{enumerate}
\item
For $\sg\in \fG$, the natural homomorphism
$$\phi :\ZZ_p[\TELs]\rightarrow\sh_{\ZZ_p,\lambda} (\sg)$$ is
an isomorphism of graded $\ZZ_p$-algebras. In other words,
$\sh_{\ZZ_p,\lambda} (\sg)=\ZZ_p[\TEL]$.
In particular, there is a natural isomorphism
\begin{equation}
\ZZ_p \TEL^{(n)} \cong \sh_{\ZZ_p,\lambda}(S)^{(n)}/\Big ( \sum_{i=1}^{n-1} \sh_{\ZZ_p,\lambda}(S)^{(i)}\sh_{\ZZ_p,\lambda}(S)^{(n-i)}\Big).
\mlabel{eq:natiso}
\end{equation}
Further, the homogeneous component $\TEL^{(n)}$ of $\TEL$ of degree $n$ has cardinality
$|\Lyn(S)^{(n)}|$, $n\geq 1$.
\mlabel{it:pgrad}
\item
For a semigroup $\sg\in \jG$, the natural homomorphism
$$\phi :\ZZ_p[\EETLs]/\bar{I}\rightarrow\sh_{\ZZ_p,\lambda} (\sg)$$ is
an isomorphism.
\mlabel{it:iso}
\end{enumerate}
\mlabel{thm:isomor}
\end{theorem}

\begin{proof}
Let $S\in \fG$ or $\jG$. By Proposition
\ref{prop:surj}, $\phi $ is surjective. By Theorem \ref{thm:pmsh},
$\phi \ot \FF_p $ is an isomorphism.
Note that for $S\in \fG$ (resp. $S\in \jG$), $\ZZ_p[\TELs]$ (resp.
$\ZZ_p[\EETLs]/\bar{I}$) is a $\ZZ_p$-module separated for the
$p$-adic topology and that $\sh_{\ZZ_p,\lambda} (\sg)$ is a free
$\ZZ_p$-module. Applying Lemma~\mref{lem:inj} with
$M=\ZZ_p[\TELs]$ (resp. $M=\ZZ_p[\EETLs]/\bar{I}$) and $N=\sh_{\ZZ_p,\lambda} (\sg)$ we obtain the
injectivity of $\phi $.

This proves Item (\mref{it:iso}) and a part of Item~(\mref{it:pgrad}).  To finish the proof of Item~(\mref{it:pgrad}), let $\sg\in \fG$.
By Lemma~\mref{lem:grad}.(\mref{it:grad2}), the algebra isomorphism $\phi $ is graded. Since the grading on $\TELs$ is obtained from $\TEL$, $\sh_{\ZZ_p,\lambda} (\sg)=\ZZ_p[\TEL]$ as a graded algebra.
Thus $\sh_{\ZZ_p,\lambda} (\sg)^{(n)}=\ZZ_p[\TEL]^{(n)}, n\geq 0$. So by Lemma~\mref{lem:grad}.(\mref{it:grad1}) we have
\begin{eqnarray*}
\ZZ_p \TEL^{(n)} &\cong& \ZZ_p[\TEL]^{(n)}/\Big ( \sum_{i=1}^{n-1} \ZZ_p[\TEL]^{(i)}\ZZ_p[\TEL]^{(n-i)}\Big)\\
&= &\sh_{\ZZ_p,\lambda}(S)^{(n)}/\Big ( \sum_{i=1}^{n-1} \sh_{\ZZ_p,\lambda}(S)^{(i)}\sh_{\ZZ_p,\lambda}(S)^{(n-i)}\Big).
\end{eqnarray*}

Since
$$ \TEL=\{ u^{\ot p^i}\ |\ u\in \EL, i\geq 0\}, \quad
\Lyn = \{ u^{\comp{p^i}}\ |\ u\in \EL, i\geq 0\}$$
by Lemma~\mref{lem:iff}.(\mref{it:iff3}), and
$$ \deg(u^{\ot p^i})=p^i \deg(u) = \deg(u^{\comp{p^i}})$$
by Eq.~(\mref{eq:degp}), we have
$ |\TEL^{(n)}| = |\Lyn^{(n)}|.$
\end{proof}

\subsection{Free Rota-Baxter algebras with coefficients in $\ZZ_p$}
\mlabel{ss:rbazp}

\begin{theorem}
Let $X$ be a finite set and let $S$ be the free abelian semigroup generated by $X$. Let $\TEL=\TEL(S)$. Let $\lambda\in \ZZ_p$ be a $p$-adic unit. Then there is a canonical subalgebra of $\sha_{\ZZ_p,\lambda}(\ZZ_p[X])$ that is isomorphic to $\ZZ_p[X\cup \TELs]$.
\mlabel{thm:rbazp}
\end{theorem}
\begin{proof}
By Theorem~\mref{thm:isomor}, $\sh_{\ZZ_p,\lambda} (S)\cong\ZZ_p[\TELs]$. The inclusion of $S$ into the free abelian monoid $M\com(X)$ induces the inclusion
$\sh_{\ZZ_p,\lambda} (S)\subseteq \sh_{\ZZ_p,\lambda}  (M\com(X))$. Then we have
$$
\begin{aligned}
\ZZ_p[X\cup \TELs]& \cong\ZZ_p[X]\ot \ZZ_p[\TELs]
\cong\ZZ_p[X]\ot \sh_{\ZZ_p,\lambda} (S) \\
& \subseteq \ZZ_p[X]\ot \sh_{\ZZ_p,\lambda}  (M\com(X))
= \sha_{\ZZ_p,\lambda}(\ZZ_p[X]).
\end{aligned}$$
\end{proof}

\section{Structure theorems on $\ZZ$}
\mlabel{sec:int}
We now study mixable shuffle algebras with coefficients in $\ZZ$ by generalizing the work of Hazewinkel~\mcite{Ha} on the Ditters Conjecture (Theorem~\mref{thm:qlyn}.(\mref{it:dh}). We first extract from his proof a general principle (Theorem~\mref{thm:grfr}) showing that a compatible system of local polynomial conditions implies a global one. This result will then be combined with our result on the local case in Section~\mref{sec:zp} and be applied to mixable shuffle algebras and free commutative Rota-Baxter algebras.

\subsection{Mixable shuffle algebras with coefficients in $\ZZ$}
\mlabel{ss:shzp}

The following lemma is well-known but we include a short proof for the lack of references.

\begin{lemma}
\begin{enumerate}
\item
A finitely generated abelian group $M$ is free of rank $k$ if $M\ot \ZZ_p\cong \ZZ_p^k$ for all prime numbers $p$.
\mlabel{it:mod}
\item
A homomorphism of finitely generated
free abelian groups $f: M_1\rightarrow M_2$ is injective and identifies $M_1$ with a direct summand of $M_2$ if for every prime $p$, the homomorphism $f\ot \ZZ_p: M_1\otimes
\ZZ_p\rightarrow M_2\otimes \ZZ_p$ is injective and identifies $M_1\ot \ZZ_p$ with a direct summand of $M_2\ot \ZZ_p$ as a $\ZZ_p$-module.
\mlabel{it:mod2}
\end{enumerate}
\mlabel{lem:bas}
\end{lemma}
\begin{proof}
(\mref{it:mod}) follows from the fundamental theorem of finitely generated abelian groups.
\smallskip

\noindent
(\mref{it:mod2}).
Since $f\ot \ZZ_p$ is injective, $\ker(f\ot \ZZ_p)=\ker(f) \ot \ZZ_p$ is the free $\ZZ_p$-module of rank 0. Thus by Item (\mref{it:mod}), $\ker f$ is the free abelian group of rank 0, so is 0.
Since $(f\ot \ZZ_p)(M_1\ot \ZZ_p)$ is a direct summand of the free $\ZZ_p$-module $M_2\ot \ZZ_p$, the quotient $(M_2\ot \ZZ_p)/(f\ot \ZZ_p)(M_1\ot \ZZ_p)=(M_2/f(M_1))\ot \ZZ_p$ is a free $\ZZ_p$-module whose $\ZZ_p$-rank is
$$\rank_{\ZZ_p}(M_2\ot \ZZ_p)-\rank_{\ZZ_p}(M_1\ot \ZZ_p)
=\rank_\ZZ(M_2)-\rank_\ZZ(M_1).$$
So by Item (\mref{it:mod}), $M_2/f(M_1)$ is free.
Therefore, $f(M_1)$ is a direct summand of $M_2$.
\end{proof}

In the following theorem, we denote $\Spec(\ZZ)=\{0\}\cup \{p\ |\ p {\rm\ a\ prime\ of\ }\ZZ\}$. Also denote $\ZZ_0=\QQ$ for ease of notations.
\begin{theorem}
Let $R=\oplus_{n\geq 0} R^{(n)}$ be a commutative graded $\ZZ$-algebra with each homogenous piece $R^{(n)}$ a free $\ZZ$-module.
Suppose that, for each $\ell\in \Spec(\ZZ)$,  there exists a graded subset $Y_{\ell}=\coprod_{n\geq 1}Y_{\ell}^{(n)}$ of $R\ot \ZZ_\ell$ with the following properties.
\begin{enumerate}
\item
For a fixed $n\ge 0$, $|Y_{\ell}^{(n)}|$ is finite with the same cardinality when $\ell\in \Spec(\ZZ)$ varies;
\mlabel{it:fin}
\item
For every $\ell\in \Spec(\ZZ)$, $R\ot \ZZ_\ell=\ZZ_\ell[Y_\ell]$ as a graded $\ZZ_\ell$-algebra.
\mlabel{it:poly}
\end{enumerate}
Then there is a graded subset $Y=\coprod_{n\ge 0} Y^{(n)}$ of $R$ such that
\begin{enumerate}
\item[(i)]
$|Y^{(n)}|=|Y_{0}^{(n)}|$ for all $n\geq 0$;
\item[(ii)]
$R\cong \ZZ[Y]$ as a graded algebra.
\end{enumerate}
\mlabel{thm:grfr}
\end{theorem}

\begin{proof}
Fix $n\ge 1$. Consider the right exact sequence
\begin{equation}
\bigoplus_{j=1}^{n-1}(R^{(j)}\ot R^{(n-j)}) \ola{\mu_n} R^{(n)} \ola{\pi_n} G^{(n)}\to 0
\mlabel{eq:zcok}
\end{equation}
where $\mu_n$ is the multiplication map and $G^{(n)}$ is the cokernel of $\mu_n$.
For any $\ell\in \Spec(\ZZ)$, by Property~(\mref{it:poly}) and the right exactness of tensoring with $\ZZ_\ell$, we obtain the right exact sequence
\begin{equation}
\bigoplus_{j=1}^{n-1} (\ZZ_\ell[Y_\ell]^{(j)}\ot \ZZ_\ell[Y_\ell]^{(n-j)}) \ola{\mu_{n,\ell}} \ZZ_\ell[Y_\ell]^{(n)} \ola{\pi_{\ell,n}} G^{(n)}\ot \ZZ_\ell \to 0,
\mlabel{eq:lcok}
\end{equation}
where $\mu_{n,\ell}$ is again the multiplication map. By Lemma~\mref{lem:grad}.(\mref{it:grad1}) we get $\ZZ_\ell[Y_\ell]^{(n)}=\im(\mu_{n,\ell})\oplus \ZZ_\ell Y_{\ell}^{(n)}$. Thus $G^{(n)}\ot \ZZ_\ell \cong \ZZ_\ell^{|Y_{\ell}^{(n)}|}$ is a free $\ZZ_\ell$-module. By Property~(\mref{it:fin}) and Lemma~\mref{lem:bas}.(\mref{it:mod}), $G^{(n)}$ is a free abelian group of rank $|Y_{\ell}^{(n)}|$. Thus the right exact sequence in Eq.~(\mref{eq:zcok}) splits and we have
$R^{(n)}=\im(\mu_n)\oplus R^{(n)}{}'$ for a free abelian group $R^{(n)}{}'\subseteq R^{(n)}{}$ of rank $|Y_{\ell}^{(n)}|$ such that $R^{(n)}{}'\cong G^{(n)}{}$ under $\pi_n$. Let $Y^{(n)}{}$ be a $\ZZ$-basis of
$R^{(n)}{}'$, $n\geq 1$, and let $Y=\cup_{n\geq 1} Y^{(n)}{}$. Let $R''$ be the subalgebra of $R$ generated by $Y$ and let $R^{(n)}{}''=R''\cap R^{(n)}{}, n\geq 1$. Let $W=\coprod_{n\ge 1} W^{(n)}{}$ be a graded set such that $W^{(n)}{}$ is in bijection with $Y^{(n)}{}$ through a map $\tau_n:W^{(n)}{}\to Y^{(n)}{}$. Define the $\ZZ$-algebra homomorphism
$$\alpha: \ZZ[W]\to R, \quad w\mapsto \tau_n(w), w\in W^{(n)}{}, n\ge 1.$$
It is a graded algebra homomorphism since it is defined piece by piece on each homogeneous subgroup. We have $R''=\im (\alpha)$.

We next prove $R''=R$ by claiming that $R^{(n)}{}\subseteq R''$ for all $n\ge 1$ by induction on $n$. When $n=1$, $R^{(1)}=\ZZ Y^{(1)}$, so the claim is clear. Suppose $R^{(k)}{}\subseteq R''$ for $k<n$.
Then since
$$R^{(n)}{}=\im(\mu_n) + \ZZ Y^{(n)}= \big(\sum_{j=1}^{n-1} R^{(j)}R^{(n-j)}\big) + \ZZ Y^{(n)},$$
we again have $R^{(n)}{}\subseteq R''$ by the induction hypothesis.
Therefore $\alpha$ is a surjective homomorphism of graded algebras. Thus $\alpha$ restricts to give a surjection
$$ \alpha_n: \ZZ[W]^{(n)}{} \to R^{(n)}{}$$
for any $n\ge 1$.

For each $n\geq 1$, $W^{(n)}$ is in bijection with $Y^{(n)}$. Also  $Y^{(n)}$, as a $\ZZ$-basis of $R^{(n)}{}'$ of rank $|Y_0^{(n)}|$, is in bijection with $Y_0^{(n)}$. So $W\cong Y_0$ as graded sets and $\QQ[W]\cong \QQ[Y_0]$ as graded algebras. Hence $\QQ[Y_0]^{(n)}$ has the same $\QQ$-dimension as that of $\QQ[W]^{(n)}$. Also by Property~(\mref{it:poly}),
$R^{(n)}{}\ot \QQ$ has the same $\QQ$-dimension as that of $\QQ[Y_0]^{(n)}{}$. Therefore $R^{(n)}{}\ot \QQ$ and $\QQ[W]^{(n)}{}$ have the same $\QQ$-dimension.
Thus the free abelian groups $R^{(n)}{}$ and $\ZZ[W]^{(n)}{}$ have the same rank.
Thus $\alpha_n$ is an isomorphism for every $n\ge 1$ and hence $\alpha$ is an isomorphism.
\end{proof}

\begin{theorem}
Let $S$ be a finitely generated free abelian semigroup. Then
for $\lambda=\pm 1$, $\sh_{\ZZ,\lambda}(S)$ is a polynomial algebra $\ZZ[Y]$, where $Y=\coprod_{n\geq 1}Y^{(n)}{}$ is a graded set whose homogeneous component $Y^{(n)}{}$ has cardinality $|\Lyn(S)^{(n)}{}|$.
Here $\Lyn(S)^{(n)}{}$ is the set of Lyndon words on $S$ of degree $n$.
\mlabel{thm:intfr}
\end{theorem}

\begin{proof}
We apply Theorem~\mref{thm:grfr} to the graded algebra $R=\sh_{\ZZ,\lambda}(S)$ where the grading is defined by the degree on words in~Eq.(\mref{eq:shdeg}).
For $\ell\in \Spec (\ZZ)$, define
$$ Y_\ell=\left \{\begin{array}{ll}
    \Lyn(S), & \ell = 0,\\
    \TEL(S)(\ell), & \ell \neq 0.
    \end{array} \right . $$
with their grading restricted from $\sh_{\ZZ,\lambda}(S)$. Then by Theorem~\mref{thm:msq2}, $R\ot \QQ \cong\QQ[Y_0]$ as graded algebras. By Theorem~\mref{thm:isomor}, for $\ell \neq 0$,
$R\ot \ZZ_\ell = \ZZ_\ell[Y_\ell]$ as graded algebras. Further, by Theorem~\mref{thm:isomor}.(\mref{it:pgrad}) and its proof,
$|Y_\ell^{(n)}|=|Y_0^{(n)}|, n\geq 1.$
Then our proof is completed by Theorem~\mref{thm:grfr}.
\end{proof}

\subsection{Weight $\lambda$ mixable shuffle algebras for countably generated free abelian semigroups}

We now extend Theorem~\mref{thm:intfr} to the countably infinite generators.

\begin{theorem} Let $X$ be a countable set. Let $F(X)$ be the free abelian semigroup
generated by $X$. Then the algebra $\sh_{\ZZ,\lambda}(F(X))$, $\lambda=\pm
1$, is a polynomial $\ZZ$-algebra.
\mlabel{thm:intfr2}
\end{theorem}
\begin{proof}
We denote $S=F(X)$ in this proof.
First we fix an order on $X$ such that
$X=\{x_1, x_2,x_3,\cdots \}$ with $x_1< x_2<x_3<\cdots$. Then we define a
degree and an order on $S$ as before. 
For every $k\geq 1$ we write $X_k=\{x_1,\cdots, x_k\}$ and let $S_k$
be the free abelian semigroup generated by $X_k$ that can be
considered as a subgroup of $S$. Then we form a direct system
$\{\sh_{\ZZ,\lambda}(S_k)\}_{k\geq 1}$ and we have
$$\sh_{\ZZ,\lambda}(S)=\dirlim
\sh_{\ZZ,\lambda}(S_k).$$

By Theorem \ref{thm:intfr} for every $k\geq 1$, $\sh_{\ZZ,\lambda}(S_k)$
is a graded polynomial algebra
\begin{equation}
\sh_{\ZZ,\lambda}(S_k)=\ZZ\Big[\coprod_{n\geq 1} Y_k^{(n)}\Big], \mlabel{eqn:poly}
\end{equation} where $Y_k^{(n)}$ is a lifting of
a basis of the quotient
$$G_k^{(n)}=\sh_{\ZZ,\lambda}(S_k)^{(n)}/\sum\limits_{1\leq i< n}
\sh_{\ZZ,\lambda}(S_k)^{(i)}\sh_{\ZZ,\lambda}(S_k)^{(n-i)}$$
to $\sh_{\ZZ,\lambda}(S_k)^{(n)}$. Let $\pi^{(n)}_k$
denote the quotient map $$\sh_{\ZZ,\lambda}(S_k)^{(n)}\rightarrow
G_k^{(n)}.$$

For our purpose we need to choose a special lifting $Y_k^{(n)}$ so that $\{Y_k^{(n)}\}_{k\geq 1}$ form an increasing sequence of subsets
for every fixed $n$. For this we
need the following lemma.

\begin{lemma}
For $n,k \geq 1$, $G^{(n)}_{k}$ is a direct summand of
$G^{(n)}_{k+1}$. \mlabel{lem:dirsum}
\end{lemma}
\begin{proof}  For a fixed prime $\ell$, we have the following commutative diagram
\[\xymatrix{
\ZZ_\ell\TEL(\ell)(S_k)^{(n)}
\ar[rr]\ar[d]^{\simeq}_{\pi^{n}_{\ell,k}} & &
\ZZ_\ell\TEL(\ell)(S_{k+1})^{(n)}
\ar[d]^{\simeq}_{\pi^{n}_{\ell,k+1}}
\\ G^{(n)}_k \ot \ZZ_\ell\ar[rr] & &  G^{(n)}_{k+1} \ot\ZZ_\ell. }\] The two vertical reduction maps are isomorphisms by Eq.~(\mref{eq:natiso}).
The homomorphism in the top row of the above diagram is induced
by the inclusion of sets
$$ \TEL(\ell)(S_k)^{(n)} \rightarrow \TEL(\ell)(S_{k+1})^{(n)}$$ and hence is
injective and identifies the source with a direct summand of the
target. Then the homomorphism
$G^{(n)}_k \ot \ZZ_\ell\to   G^{(n)}_{k+1}\ot\ZZ_\ell$ in the bottom row is also injective
and identifies $ G^{(n)}_k\ot\ZZ_\ell$ with a direct summand of $G^{(n)}_{k+1}\ot \ZZ_\ell $.
By Lemma~\mref{lem:bas}.(\mref{it:mod2}) we obtain the lemma.
\end{proof}

Now we choose our $Y_k^{(n)}$ by induction on $k\geq 1$. We first fix a lifting $Y^{(n)}_1$. For a given $k\geq 1$, suppose we have chosen
$Y^{(n)}_k$. Then $\pi^{(n)}_k(Y^{(n)}_k)$ is a basis of
$G^{(n)}_{k}$. By Lemma \mref{lem:dirsum}, $G^{(n)}_{k}$ is a direct
summand of $G^{(n)}_{k+1}$. In other words, we may write
$$G^{(n)}_{k+1}=G^{(n)}_k\oplus G'^{(n)}_{k+1}.$$ Then
$G'^{(n)}_{k+1}$ is a free abelian group and let $\bas'^{(n)}_{k+1}$
be a basis of $G'^{(n)}_{k+1}$. Let $Y'^{(n)}_{k+1}$ be a lifting of
$\bas'^{(n)}_{k+1}$ to $\sh_{\ZZ_p,\lambda}(S_{k+1})^{(n)}$. Then we can define $Y^{(n)}_{k+1}$ to be the
disjoint union $Y^{(n)}_k \coprod Y'^{(n)}_{k+1}$ since
$\pi^{(n)}_{k+1}(Y^{(n)}_k \coprod Y'^{(n)}_{k+1})$ is a basis of
the free abelian group $G^{(n)}_{k+1}$. This completes the
induction.

Let $$Y_k=\coprod_{n\geq 1}Y_k^{(n)}.$$ Then $Y_k$ is a subset of
$Y_{k+1}$. From the construction and Eq.~(\mref{eqn:poly}) we obtain
$$\sh_{\ZZ,\lambda}(S_k)=\ZZ[Y_k].$$ Therefore
$$\sh_{\ZZ,\lambda}(S)=\dirlim \sh_{\ZZ,\lambda}(S_k)=\dirlim \ZZ [Y_k]=\ZZ [\cup_{k\geq 1} Y_k] $$
is a polynomial $\ZZ$-algebra, as expected.
\end{proof}

\subsection{Free commutative Rota-Baxter algebras with coefficients in $\ZZ$}
\mlabel{ss:rbaz}

\begin{theorem}
Let $X$ be a at most countably many set. Let $F(X)$ be the free abelian semigroup generated by $X$. Let $\lambda =\pm 1$. Then there is a set $\Omega$ of variables such that
\begin{equation}
\sha_{\ZZ,\lambda}(\ZZ[X])\cong \ZZ[\Omega] \oplus N
\mlabel{eq:summ}
\end{equation}
where $N=N_S$ is the subgroup of $\sha_{\ZZ,\lambda}(\ZZ[X])$ spanned by pure tensors of the form
$$w_0\ot \cdots \ot w_r, w_i\in \{1\}\cup F(X), 1\leq i\leq r, w_i=1 \text{\  for some } 1\leq i\leq r, r\geq 1.$$
When $X$ is finite. Then $\Omega=X\cup Y$, where $Y$ is a graded set in bijection with the graded set $\Lyn(F(X))$ of Lyndon words. \mlabel{thm:rbaz}
\end{theorem}

\begin{proof}
By Theorem~\mref{thm:intfr} and Theorem~\mref{thm:intfr2}, $\sh_{\ZZ,\lambda}(F(X))=\ZZ_p[Y]$ for a set $Y$ of variables.
Let $M\com(X)$ be the free commutative monoid generated by $X$. Then $M\com(X)=\{1\}\cup F(X)$. So a word in $\sh_{\ZZ,\lambda}(M\com(X))$ is of the form $w=w_1\ot \cdots \ot w_r$ where either each $w_i$ is in $F(X)$ or at least one of $w_i$ is 1. A word $w$ is in $\sh_{\ZZ,\lambda}(F(X))$ precisely when it is of the first form. We denote $N^+$ to be the subgroup of $\sh_{\ZZ,\lambda}(M\com(X))$ generated by elements of the second form. Then we have
$$\sh_{\ZZ,\lambda}(M\com(X))=\sh_{\ZZ,\lambda}(F(X))\oplus N^+ \cong \ZZ[Y]\oplus N^+.$$
Since $\ZZ[X]=\ZZ M\com(X)$, we have
$$\begin{aligned}
\sha_{\ZZ,\lambda}(\ZZ[X])&=\ZZ[X]\ot \sh_{\ZZ,\lambda}(M\com(X)) \cong \ZZ[X] \ot (\sh_{\ZZ,\lambda}(F(X))\oplus N^+) \\
& \cong \ZZ[X]\ot (\ZZ[Y]\oplus N^+) \cong \ZZ[X\cup Y] \oplus (\ZZ[X]\ot N^+).
\end{aligned}
$$
Then we just need to take $\Omega=X\cup Y$ and $N=\ZZ[X]\ot N^+$ to get the direct sum decomposition in Eq.~(\mref{eq:summ}).

When $X$ is finite, by Theorem~\mref{thm:isomor}, we have $Y$ in the specified form as prescribed.
\end{proof}

As a final note, we elaborate on the significance of Theorem~\mref{thm:rbaz}.
By Theorem~\mref{thm:rbl}, $\sha_{\QQ,\lambda}(A(X))$ is a polynomial $\QQ$-algebra generated by $\overline{\Lyn}(X):=X\cup \{1\ot w\ |\ w\in \Lyn(M\com(X))\}.$ Since $\overline{\Lyn}(X)$ is a part of a $\ZZ$-basis of $\sha_{\ZZ,\lambda}(A(X))$, it follows that $\overline{\Lyn}(X)$ generates a polynomial $\ZZ$-subalgebra of $\sha_{\ZZ,\lambda}(A(X))$. There is no inclusion relation between the polynomial generating set $Y$ in Theorem~\mref{thm:rbaz} and $\overline{\Lyn}(X)$ in Theorem~\mref{thm:rbl} since $Y$ is not the set of Lyndon words, only in bijection with this set. However, the polynomial subalgebra $\ZZ[X\cup Y]$ in Theorem~\mref{thm:rbaz} can be more useful in studying the structure of $\sha_{\ZZ,\lambda}(A(X))$ because of the direct sum decomposition in Eq.~(\mref{eq:summ}). This is similar to the importance of studying direct summands of abelian groups. It is easy to obtain free subgroups in a torsion-free abelian group, such as $\QQ$, but it is more useful to obtain a direct summand that is free. Similarly, there are many polynomial subalgebras in
a free commutative Rota-Baxter algebra $\sha(A)$, but it is more useful to have such a subalgebra that is also a direct summand. For example, in $\sha_{\ZZ,0}(\ZZ)$ which is just the divided power algebra $\oplus_{n\geq 0} \ZZ x_n$ with $x_m x_n=\binc{m+n}{m} x_{m+n}$,
the subalgebra generated by any $f\not\in \ZZ$ is a polynomial algebra, but the algebra itself is not a polynomial algebra, none does it have a polynomial subalgebra as a direct summand. In the case we consider, it would be interesting to find out whether the polynomial algebra summand in Eq.~(\mref{eq:summ}) can be extended to a larger such summand.

%
%

\end{document}